%% file: main.tex
\documentclass[reqno, 10pt]{amsart}

\input{Latex_Code/Packages}

\input{Latex_Code/Commands}

\input{Latex_Code/Environments}

\title[Bounds on Sobolev norms for fractional Schr\"{o}dinger equations]{Bounds on the growth of high Sobolev norms of solutions to the fractional nonlinear Schr\"{o}dinger equation on $\R^2$ and $\R^3$}

\author[A. Kapetanios]{Antonios Kapetanios}
\address{Mathematics Institute, Zeeman Building, University of Warwick, Coventry, CV4 7AL, United Kingdom}
\email{antonios.kapetanios@warwick.ac.uk}
\author[J. L. Rodrigo]{Jos\'{e} L. Rodrigo}
\address{Mathematics Institute, Zeeman Building, University of Warwick, Coventry, CV4 7AL, United Kingdom}
\email{j.rodrigo@warwick.ac.uk}
\author[V. Sohinger]{Vedran Sohinger}
\address{Mathematics Institute, Zeeman Building, University of Warwick, Coventry, CV4 7AL, United Kingdom}
\email{v.sohinger@warwick.ac.uk}

\date{}

\subjclass[2020]{35Q55} 
\keywords{fractional nonlinear Schr\"{o}dinger equation, cubic nonlinearity, Hartree-type nonlinearity, growth of Sobolev norms, resonant decomposition, upside-down I-method}

\begin{document}  

\input{Text/Abstract}

\maketitle

\input{Text/Main_Text}


\input{Text/Bibliography}
\end{document}

%% file: Latex_Code/Packages.tex
\usepackage{amsmath}
\usepackage{amssymb}
\usepackage{amsthm}
\usepackage{mathcomp}
\usepackage{moreenum}
\usepackage{textcomp}
\usepackage{mathtools}
\usepackage{faktor} 
\usepackage{thm-restate}

\usepackage{enumitem}

\usepackage{cases}

\usepackage{titletoc}

\usepackage{dsfont}

\usepackage{xcolor}

\usepackage{url}

 \usepackage[letterpaper,left=1in,right=1in,top=1.1in,bottom=1.1in]{geometry}

\usepackage[hidelinks,colorlinks=true, linkcolor=blue,citecolor=blue, urlcolor=blue]{hyperref}

\usepackage[super]{nth}

\usepackage{comment}

\usepackage{mathabx}

\usepackage{empheq}

\usepackage{cleveref}

%% file: Latex_Code/Commands.tex
\newcommand{\R}{\mathbb{R}}

\newcommand{\N}{\mathbb{N}}
\newcommand{\C}{\mathbb{C}}

\newcommand{\F}{\mathcal{F}}

\newcommand{\T}{\mathbb{T}}

\newcommand{\D}{\mathcal{D}}

\newcommand{\abs}[1]{\left| #1 \right|}

\newcommand{\pt}[1]{\left( #1 \right)}

\newcommand{\ra}{\rightarrow}

\usepackage{scalerel,stackengine}
\stackMath
\newcommand\reallywidehat[1]{%
\savestack{\tmpbox}{\stretchto{%
   \scaleto{%
\scalerel*[\widthof{\ensuremath{#1}}]{\kern-.6pt\bigwedge\kern-.6pt}%
     {\rule[-\textheight/2]{1ex}{\textheight}}
   }{\textheight}%
}{0.5ex}}%
\stackon[1pt]{#1}{\tmpbox}%
}
\newcommand\reallywidecheck[1]{%
\savestack{\tmpbox}{\stretchto{%
   \scaleto{%
\scalerel*[\widthof{\ensuremath{#1}}]{\kern-.6pt\bigwedge\kern-.6pt}%
     {\rule[-\textheight/2]{1ex}{\textheight}}
   }{\textheight}%
}{0.5ex}}%
\stackon[1pt]{#1}{\scalebox{-1}{\tmpbox}}%
}
\newcommand\reallywidetilde[1]{%
\savestack{\tmpbox}{\stretchto{%
   \scaleto{%
\scalerel*[\widthof{\ensuremath{#1}}]{\kern-.6pt\sim\kern-.6pt}%
     {\rule[-\textheight/2]{1ex}{\textheight}}
   }{\textheight}%
}{0.5ex}}%
\stackon[1pt]{#1}{\tmpbox}%
}

%% file: Latex_Code/Environments.tex
\theoremstyle{definition}
\newtheorem{theorem}{Theorem}[section]
\newtheorem{proposition}[theorem]{Proposition}
\newtheorem{lemma}[theorem]{Lemma}
\newtheorem{corollary}[theorem]{Corollary}

\newtheorem{remark}[theorem]{Remark}



\makeatletter
\renewcommand{\section}{\@startsection{section}{1}%
  \z@
  {2.0\baselineskip \@plus .3\baselineskip \@minus .2\baselineskip}
  {1.0\baselineskip \@plus .2\baselineskip}
  {\large\normalfont\centering\scshape}
}
\makeatother


%% file: Text/Abstract.tex
\begin{abstract}
 We prove polynomial bounds on the growth of high Sobolev norms of solutions to the defocusing fractional nonlinear Schr\"{o}dinger equation with cubic and Hartree-type nonlinearities in two and three dimensions. Our result is based on the \textit{upside-down I-method} \cite{CKSTT_Polynomial, CollianderKwonOh, Sohinger_S1, Sohinger_R, Sohinger_Hartree_2D} and the \textit{method of higher modified energies}. The analysis of the fractional case in higher dimensions is possible due to a higher-dimensional analogue of a \emph{resonance inequality}, first proved in \cite{DemirbasErdoganTzirakis} in one dimension, which allows us to control the nonresonant frequency contributions. 
\end{abstract}

%% file: Text/Main_Text.tex
\vspace{-1cm}

\section{Introduction.} 

\subsection{Statement of the problem and main results.}
In this paper, we consider the initial value problem (IVP) for the defocusing fractional nonlinear Schr\"{o}dinger equation (NLS) with cubic nonlinearity:
\begin{equation} \label{cubic_fractional_NLS}
\begin{cases}
\begin{aligned}
&i u_t - (-\Delta)^\alpha u = |u|^2u, \quad  & t\in\R, &\quad x\in \R^d, \\ 
&u(0,\cdot) =u_0 \in H^s(\R^d),  
\end{aligned}
\end{cases}
\end{equation}
where $\alpha \in (\frac{1}{2},1]$, and $\reallywidehat{(-\Delta)^\alpha u}(\xi) := |\xi|^{2\alpha} \widehat u(\xi)$. The IVP (\ref{cubic_fractional_NLS}) possesses conserved mass and energy, namely
\begin{align*} 
M(u(t)) &:= \| u(t)\|_{L^2}^2, \qquad\\[6pt]
E(u(t)) &:= \frac12\|(-\Delta)^{\alpha/2}u\|_{L^2}^2+\frac14\int_{\R^d} |u(t,x)|^4 \; dx. \qquad
\end{align*}

The $\dot{H}^s$ critical exponent of (\ref{cubic_fractional_NLS}) is
\[
s_c := \frac{d}{2}- \alpha.
\]
Local well-posedness of (\ref{cubic_fractional_NLS}) is established in \cite{ CazenaveBook, HongSire} for $s > s_c$, while global well-posedness of (\ref{cubic_fractional_NLS}) in $H^\alpha$ follows by energy conservation whenever
\begin{equation*} \label{alpha_condition_scaling}
 s_c < \alpha \quad \iff \quad \alpha > \frac{d}{4}.   
\end{equation*}
By persistence of regularity, we can further deduce global well-posedness in $H^s$ for any $s>\alpha$.  

Throughout this manuscript, we assume that
\begin{equation} \label{alpha_s_conditions}
   d \in \{2,3\}, \qquad \alpha \in (\frac{d}{4},1], \qquad \text{ and } \qquad s> \frac{d}{2}.
\end{equation}

The assumption $s>\frac{d}{2}$ is used later to prove local-in-time bounds for the solution  of (\ref{cubic_fractional_NLS})  (see Proposition \ref{Proposition_3.1_d.alpha}). Under assumptions  (\ref{alpha_s_conditions}), we always have  $s>\alpha$ since $\alpha \in (\frac{1}{2},1]$, and hence a global solution to (\ref{cubic_fractional_NLS}) exists throughout the range considered.

In this paper, we obtain polynomial bounds on the growth of high Sobolev norms for solutions to (\ref{cubic_fractional_NLS}) in the range (\ref{alpha_s_conditions}); that is,  for $s>1$ and $\alpha \in (\frac{1}{2},1]$ in two dimensions, and for $s>\frac{3}{2}$ and $\alpha \in (\frac{3}{4},1]$ in three dimensions. Our main result is as follows.
\begin{theorem} \label{Corollary_cubic_fractional_NLS}
Let $d \in \{2,3\}$ and let $u$ be the global solution of (\ref{cubic_fractional_NLS}) for $s>\frac{d}{2}$ and $\alpha \in (\frac{d}{4},1]$.
 Then, there exists a function $C_{s}=C_{s}(d,u_0)$, continuous on $H^\alpha(\R^d)$, such that for all $t\in\R$,\footnote{We write $x+$ for $x+\epsilon$ with $0<\epsilon\ll 1$.}
\begin{equation*}
\|u(t)\|_{H^s} \leq C_{s} \, (1+|t|)^{ \frac{2}{3(2\alpha -1)} s+} \; \|u_0\|_{H^s}, \qquad \text{ for }  d=2, 
\end{equation*}
and
\begin{equation*}
\|u(t)\|_{H^s} \leq C_{s}\, (1+|t|)^{ \frac{1}{4\alpha-3 } s+} \; \|u_0\|_{H^s}, \qquad \text{ for }  d=3.
\end{equation*}
\end{theorem}

Our analysis works more generally for the defocusing fractional NLS with Hartree-type nonlinearity, i.e.
\begin{equation}\label{equation_fractional_hartree}
\begin{cases}
\begin{aligned}
&i u_t - (-\Delta)^\alpha u = (V * |u|^2)\,u, \quad  & t\in\R, &\quad x\in \R^d, \\ 
&u(0,\cdot) =u_0 \in H^s(\R^d),  
\end{aligned}
\end{cases}
\end{equation}
where $\alpha \in (\frac{1}{2},1]$, and $V$ is an interaction potential satisfying $V\in L^1(\R^d),  V\geq 0$, and $V(x) = V(-x)$. The IVP (\ref{equation_fractional_hartree}) has conserved mass and energy, namely
\begin{align*} 
M(u(t)) &:= \| u(t)\|_{L^2}^2, \qquad\\[6pt]
E(u(t)) &:= \frac12\|(-\Delta)^{\alpha/2}u\|_{L^2}^2+\frac14\int_{\R^d} (V *|u|^2)(t,x) \, |u(t,x)|^2 \; dx. \qquad
\end{align*}

Local well-posedness of (\ref{equation_fractional_hartree}) can be obtained as in \cite{CazenaveBook, HongSire} for $s>s_c$, while global well-posedness for $s>\alpha$ follows by energy conservation and persistence of regularity. The main result for   (\ref{equation_fractional_hartree}) is as follows.

\begin{theorem} \label{Main_Theorem}
Let $d \in \{2,3\}$ and let $u$ be the global solution of (\ref{equation_fractional_hartree}) for $s>\frac{d}{2}$ and $\alpha \in (\frac{d}{4},1]$.
 Then, there exists a function $C_{s}=C_{s}(d,u_0)$, continuous on $H^\alpha(\R^d)$, such that for all $t\in\R$,
\begin{equation*}
\|u(t)\|_{H^s} \leq C_{s} \,(1+|t|)^{ \frac{2}{3(2\alpha -1)} s+} \; \|u_0\|_{H^s},  \qquad \text{ for }  d=2,
\end{equation*}
and
\begin{equation*}
\|u(t)\|_{H^s} \leq C_{s}\, (1+|t|)^{ \frac{1}{4\alpha-3 } s+} \; \|u_0\|_{H^s}, \qquad \text{ for }  d=3.
\end{equation*}
\end{theorem}

The proof of Theorem \ref{Main_Theorem} still holds if one formally takes $V=\delta$, and thus we can deduce in particular Theorem \ref{Corollary_cubic_fractional_NLS} from Theorem \ref{Main_Theorem}. Therefore, in this manuscript, we only prove Theorem \ref{Main_Theorem}.

We note that in \cite{HongSire}, the authors established, in particular, small-data scattering for solutions to the fractional NLS
\begin{equation*} 
\label{p_fractional_NLS} 
\begin{cases} 
\begin{aligned} &i u_t - (-\Delta)^\alpha u = |u|^{p-1} u, \quad & t\in\R, &\quad x\in \R^d, \\ &u(0,\cdot) =u_0 \in H^s(\R^d), 
\end{aligned} \end{cases} 
\end{equation*}
for $p>3$, when $d\geq 2$. This implies, in particular, uniform bounds on the $H^s$-norm of the solutions (see \cite[Appendix B]{Sohinger_Hartree_2D} for more details). However, to the best of our knowledge, for $\alpha \in (\frac{1}{2},1)$, the same scattering result has not been obtained for $p=3$.

Moreover, unless $\alpha > \frac{d}{2}$ or $\alpha = \frac{d}{2}$ with sufficiently small $L^2$-norm of the initial datum $u_0$, our analysis does not extend to the focusing case, since the energy is not coercive for $\alpha < \frac{d}{2}$, and thus energy conservation does not imply global control of the $H^\alpha$-norm of $u$. The regime $\alpha > \frac{d}{2}$ for $d\geq 2$, however, lies outside the  interval $\frac{1}{2} < \alpha \leq 1$ that we are considering.

\subsection{Motivation for the problem and known results} 

 The fractional Schr\"{o}dinger equation was first introduced by Laskin \cite{Laskin_Levy_paths, Laskin_fSE} in the context of fractional quantum mechanics. Laskin replaced the Brownian trajectories of the Feynman path integral with the more general L\'{e}vy $2\alpha$-stable paths, $\alpha \in (\frac{1}{2},1]$, leading to the fractional Schr\"{o}dinger equation in $\R^3$. Since then, well-posedness and other aspects of the fractional Schr\"{o}dinger equation with different types of nonlinearities have been studied extensively; see for instance \cite{ChoHajaiejHwangOzawa, DemirbasErdoganTzirakis, GuoSireWangZhao, GuoWang, HongSire, LeeLeeRoncal, MegretskiSkouloudis} and references therein.

Bounds on the growth of higher Sobolev norms for the NLS were first studied by Bourgain in \cite{Bourgain_Sobolev_norms}. By local well-posedness theory \cite{BourgainBook, CazenaveBook, GinibreVelo, TaoDispersive}, one can show that there exists $C, \tau >0$, depending only on the initial data $u_0$, such that
\begin{equation} \label{local_WP_bound}
    \|u(t + \tau)\|_{H^s} \leq C \| u(t) \|_{H^s},
\end{equation}
for all $t$.  Then, iterating (\ref{local_WP_bound}) yields the exponential bound
\begin{equation} \label{Mot_exponential}
\| u(t)\|_{H^s} \leq C e^{c|t|},   
\end{equation}
where $C,c >0 $ again depend only on the initial data $u_0$. Estimate (\ref{Mot_exponential}), however, can be improved for various nonlinear dispersive equations. For instance, in \cite{Bourgain_Sobolev_norms}, Bourgain improved (\ref{local_WP_bound}) to 
\[
\|u(t + \tau)\|_{H^s} \leq \|u(t)\|_{H^s} + C\|u(t)\|_{H^s}^{1-\delta},
\]
which then implies the polynomial bound
\begin{equation} \label{Mot_Bourgain_implied}
\|u(t)\|_{H^s} \leq C (1+|t|)^{\frac{1}{\delta}}.    
\end{equation}

 In later years, Staffilani \cite{Staffilani_KdV} showed the same norm bound (\ref{Mot_Bourgain_implied}) for solutions to the cubic NLS and the Korteweg–de Vries (KdV) equation, using multilinear estimates in $X^{s,b}$-spaces.  In \cite{Bourgain_Sobolev_norms_Schrodinger}, Bourgain studied the Schr\"{o}dinger equation
 \begin{equation} \label{Mot_Bourgain_Schrodinger}
  i u_t + \Delta u = Vu    
 \end{equation}
for a sufficiently regular potential $V$, and showed that solutions to (\ref{Mot_Bourgain_Schrodinger}) satisfy
\[
\| u(t)\|_{H^s} \leq C (1+|t|)^{\epsilon},   
\]
for any $\epsilon>0$. 

For Hartree-type nonlinearities,  the third author \cite{Sohinger_Hartree_2D} studied (\ref{equation_fractional_hartree}) with $s>1$ in the case $\alpha =1$ on $\T^2$ and on $\R^2$.  In the nonperiodic  case, in particular, he  obtained 
\[
\| u(t)\|_{H^s} \leq C (1+|t|)^{\frac{4}{7}s+}.
\]

In $\R^2$, Dodson \cite{Dodson_R2} proved scattering of solutions to (\ref{cubic_fractional_NLS}) with $\alpha =1$ for $s=0$. By the arguments in \cite[Appendix B]{Sohinger_Hartree_2D}, Dodson's result implies scattering of the solutions in $H^s$ for $s>1$. In $\R^3$,  Colliander et al. \cite{CKSTT_global_scattering} established scattering of solutions to (\ref{cubic_fractional_NLS}) with $\alpha=1$ for $s> \frac{4}{5}$. Consequently, in both cases the $H^s$-norm of the solution remains uniformly bounded in time; more precisely, there exists a constant $C>0$, depending the initial datum $u_0$, such that
\[
\sup_{t \in \R} \|u(t)\|_{H^s} \leq C,
\]
for $s>1$ in $\R^2$, and  $s>\frac{4}{5}$ in $\R^3$.

For time-dependent Schr\"{o}dinger equations, Maspero and Robert \cite{MasperoRobert}  established polynomial and logarithmic bounds for solutions of equations of the form
\[
i u_t = L(t) u(t)
\]
defined on a scale of Hilbert spaces $\{ \mathcal{H}^r\}_{r \in \R}$,
where $L(t)$ is a linear, time-dependent, and sufficiently regular operator.

In the opposite direction, Maspero \cite{Maspero} recently proved polynomial bounds from below for the class of abstract linear Schr\"{o}dinger-type equations
\begin{equation*} \label{Maspero_class}
i u_t  = K_0 u + V(t,x) u 
\end{equation*}
defined on a scale of Hilbert spaces $\{ \mathcal{H}^r\}_{r \in \R}$, where $K_0$ is a selfadjoint operator and $V$ is a smooth periodic potential. His results imply, in particular, that there exists a solution $u$ to the one-dimensional periodic half-wave equation 
\[
i u_t = (-\Delta)^{\frac{1}{2}}u + V(t,x) u, \qquad V(t,x) := \cos(t) \,v(x), \qquad v \in C^\infty(\T),
\]
satisfying
\[
\| u(t) \|_{H^r(\T)} \gtrsim (1+|t|)^r, \qquad t > T,
\]
for any $r>0$ and some $T>0$ large enough.

In the fractional case, Thirouin \cite{Thirouin} first showed polynomial bounds for the defocusing fractional cubic NLS  on the one-dimensional torus. More specifically, for $\alpha \in (\frac{1}{3},\frac{1}{2}) \cup( \frac{1}{2},1)$, he proved that
\[
\| u(t) \|_{H^{\alpha+n}(\T)} \lesssim (1+ |t|)^A, \qquad \text{ for all } \; n \in \N,
\]
for some $A$ depending on $n$ and $\alpha$. In later years, Erdo\u{g}an,  G\"{u}rel and Tzirakis \cite{ErdoganGurelTzirakis} obtained polynomial bounds for the focusing and defocusing fractional cubic NLS (\ref{cubic_fractional_NLS}) on $\T$ and $\R$. In the nonperiodic case, they showed in particular that
\[
\| u(t)\|_{H^s} \lesssim (1+|t|)^{\max\big(1, \frac{s-\alpha}{2\alpha - 1}+\big)}  , 
\]
for $\alpha \in (\frac{1}{2}, 1]$ and $s>\alpha$. More recently, Wang \cite{Wang} established polynomial bounds for the defocusing cubic NLS on $\T^d$, $d\geq 2$, extending the results of Thirouin \cite{Thirouin} to higher dimensions.

\subsection{Methodology} Our techniques are based on the \textit{upside-down I-method} \cite{CKSTT_Polynomial, CollianderKwonOh, Sohinger_S1, Sohinger_R, Sohinger_Hartree_2D} and the \textit{method of higher modified energies}. Following \cite{Sohinger_Hartree_2D}, we first define an \textit{upside-down I-operator} $\D=\D_N$, which, by definition, is a Fourier multiplier operator depending on a frequency threshold $N$, and  by construction satisfies
\[
 \| \D u\|_{L^2} \lesssim \|u\|_{H^s} \lesssim N^s \| \D u\|_{L^2}.
\]
Thus, in order to control the Sobolev norm of $u$, it is enough to control the quantity
\[
E_1(u) := \| \D u \|_{L^2}^2.
\]

We then define the \textit{modified energy} $E_2(u)$ by adding a correction term to $E_1(u)$. This quantity still satisfies $E_2(u) \sim \| \D u\|_{L^2}^2$, but it is even more slowly varying than $E_1(u)$. This allows us to obtain better polynomial bounds for the Sobolev norm of $u$. The quantities $E_1$ and $E_2$ are defined in Subsection \ref{higher_modified_energy}, and depend on a frequency threshold $N$ because of the definition of the operator $\D=\D_N$.

To bound the correction term in the definition of $E_2(u)$, we prove an appropriate \emph{resonance inequality}, which we use to control the nonresonant frequency contributions. Such an inequality was first proved to hold globally in \cite{DemirbasErdoganTzirakis} in one dimension, in contrast to the higher-dimensional case, which fails in a certain regime. The precise statement is given in Lemma \ref{Sufficient_and_necessary_nonperiodic}.

\subsection{Organization of the paper}

In Section \ref{Terminology_and_known_facts} we introduce some basic definitions and notation, and recall some known facts. In Section \ref{Estimates_Xsb} we prove linear and bilinear estimates for some function spaces, which we use throughout our analysis. Section \ref{Main_theorem_proof} is devoted to the proof of the main result, i.e. Theorem \ref{Main_Theorem}. In Section \ref{Appendix}, we prove local-in-time bounds for (\ref{equation_fractional_hartree}), a higher-dimensional resonance inequality that we use to prove Theorem \ref{Main_Theorem} in Section \ref{Main_theorem_proof}, and the Kato-Ponce inequality in $H^s\cap L^\infty$.

\subsection{Acknowledgments}

The authors would like to thank Nikolaos Tzirakis, Enno Lenzmann, Gigliola Staffilani, and Kay Kirkpatrick for several useful comments and discussions.

\subsection{Licence}

For the purpose of open access, the authors have applied a Creative Commons Attribution (CC BY) licence to any Author Accepted Manuscript version arising from this submission.

\section{Terminology and auxiliary results} \label{Terminology_and_known_facts} 
We denote by $\N$ the set $\{0,1,2, \ldots \}$.
We write $A\lesssim B$ if $A\leq C B$ for some $C>0$, and $A \ll B$ if $A\leq c B$ for some small $c>0$. If $A \lesssim B$ and $B \lesssim A$, we write $A \sim B$.  If $C$ depends on some parameter $p$, we write $C=C(p)$ or simply $A\lesssim_p B$. Given $x \in \R$, we write $x+$ and $x-$ for quantities of the form $x+\epsilon$ and $x-\epsilon$,
respectively, where $0<\epsilon\ll 1$. Moreover, we denote the Japanese bracket of $x$ by
\[
\langle x \rangle := (1+|x|^2)^{\frac{1}{2}}.
\]

We define the Fourier and inverse Fourier transforms on $\R^d$ by
\[
\F f(\xi) = \widehat{f}(\xi) := \int_{\R^d} e^{-i \, x\cdot \xi} f(x) \; dx, \qquad  \F^{-1}f(x) = \widecheck f(x) = \frac{1}{(2\pi)^d} \int_{\R^d} e^{i \, x\cdot \xi} f(x) \; dx.
\]
For $f = f(t,x) \in L^1(\R \times \R^d)$, we write $\widetilde{f}$ for the space-time Fourier transform; that is,
\[
\widetilde{f}(\tau,\xi) := \int_{\R} \int_{\R^d} e^{- i x \cdot \xi} e^{- it \tau} f(t,x) \; dx \, dt.
\]

We denote by $\dot{W}^{s,r}(\R^d)$  and $W^{s,r}(\R^d)$  the homogeneous and inhomogeneous Sobolev spaces, with norms
\[
\| f\|_{\dot{W}^{s,r}(\R^d)} := \| |\nabla|^s f \|_{L^r(\R^d)}, \qquad \text{ and } \qquad  \| f\|_{W^{s,r}(\R^d)} := \| \langle\nabla \rangle ^s f \|_{L^r(\R^d)}.
\]
For $r=2$, we simply write $\dot{H}^s(\R^d) = \dot{W}^{s,2}(\R^d)$ and  $H^s(\R^d) = W^{s,2}(\R^d)$. Additionally, we define
\[
H^\infty(\R^d) := \bigcap_{s>0} H^s(\R^d).
\]

We define the norm
\[
\| f \|_{L_t^q L_x^r} := \big \| \|f(t,x) \|_{L_x^r(\R^d)} \big \|_{L_t^q(\R)} = \pt{ \int_\R \pt{\int_{\R^d} |f(t,x)|^r  d  x}^{\frac{q}{r}} d t }^{\frac{1}{q}},
\]
with the usual modifications when $q=\infty$ or $r=\infty$. If $q=r$, we write $L_{t,x}^q := L_t^q L_x^q$ for simplicity.

If $u$ solves the initial value problem (IVP)
\begin{equation} \label{solution_dispersive_PDE_notation}
   i u_t(t,x) - (-\Delta)^{\alpha}u(t,x) = 0, \qquad u(0,x) = u_0(x),
\end{equation}
 where $\widehat{(-\Delta)^{\alpha} u}(\xi) := |\xi|^{2\alpha} \widehat{u}(\xi)$, then we denote the solution of (\ref{solution_dispersive_PDE_notation}) by
\begin{equation} \label{Salpha_definition}
  S_\alpha(t) u_0 := \F^{-1} ( e^{-it|\xi|^{2\alpha}} \widehat{u}_0(\xi)).  
\end{equation}

For $\alpha>0$  and $s,b \in \R$, we define the function spaces $X_\alpha^{s,b}(\R \times \R^d)$
 as the closure of the Schwartz functions $\mathcal{S}_{t,x}(\R \times \R^d)$ under the norm 
\begin{equation} \label{Xsb_definition}
\| u\|_{X_{\alpha}^{s,b}} := \| \langle \xi \rangle^s \langle \; \tau + |\xi|^{2\alpha}\rangle^b \; \widetilde{u}(\tau,\xi) \|_{L_\tau^2L_\xi^2}.    
\end{equation}
For simplicity, we write $X^{s,b}$ instead of $X_\alpha^{s,b}$.

A pair $(q,r)$ is called \textit{Schr\"{o}dinger admissible} if $(q,r) \in [2,+\infty]^2$, $(q,r,d) \neq (2,+\infty,2)$ and
\begin{equation} \label{Schrodinger_admissible}
     \frac{2}{q}+ \frac{d}{r} \leq \frac{d}{2}.
\end{equation}
If, in addition, equality holds in (\ref{Schrodinger_admissible}), we say that $(q,r)$ is \textit{sharp Schr\"{o}dinger admissible}.

Next, we define some multilinear expressions, which have been used in \cite{CKSTT_global_not_refined, CKSTT_global_refined}. For $k\geq 2$ an even integer, we define $M_k = M_k(\xi_1, \ldots, \xi_k)$ to be any function on the hyperplane
\[
\Gamma_k := \{(\xi_1, \ldots , \xi_k) \in (\R^d)^k : \xi_1 + \ldots + \xi_k = 0\},
\]
which we endow with the measure $d\Gamma_k(\xi)$ defined by
\[
\int_{\Gamma_k} F(\xi_1,\ldots,\xi_k)\,d\Gamma_k(\xi) := \int_{(\R^d)^{k-1}} F\left(\xi_1, \ldots, \xi_{k-1}, -\xi_1 - \ldots - \xi_{k-1}\right) \,d\xi_1 \ldots d\xi_{k-1}.
\]

We then define the $k$-linear functional $\lambda_k(M_k; f_1, \ldots, f_k)$ by 
\[
\lambda_k(M_k; f_1, \ldots, f_k) := \int_{\Gamma_k} M_k(\xi_1, \ldots, \xi_k) \; \prod_{j=1}^k \widehat{f}_j(\xi_j) \; d\Gamma_k(\xi).
\]
Moreover, as in \cite{CKSTT_global_not_refined}, we denote
\begin{equation} \label{lambda_M_f}
\lambda_k(M_k;f) := \lambda_k(M_k; f, \overline{f},\ldots, f, \overline{f}). 
\end{equation}

Given indices $i,j,k$ and $p$, we will often write 
\begin{equation} \label{xi_ijpk}
\xi_{ij} := \xi_i + \xi_j, \qquad \xi_{ijk} := \xi_i + \xi_j + \xi_k, \qquad \text{ and } \qquad \xi_{ijkp} := \xi_i + \xi_j + \xi_k + \xi_p.    
\end{equation}

Given a function $u \in L^2(\R^d)$, we define the (essential) \textit{support} of $u$ by
\[
\operatorname{supp} u := \R^d\setminus \bigcup \left\{U\subset \R^d\text{ open}: u=0 \text{ a.e. on }U
\right\}.
\]

Let $\phi_N \in C_c^\infty(\R^d)$, with $N \in 2^\N$, such that $\phi_1$ is supported in the region $\langle \xi \rangle \lesssim 1$,  $\phi_N$ is supported in the annulus $\langle \xi \rangle \sim N$, and
\[
\sum_{N \in 2^\N} \phi_N = 1.
\]
Then, for every $N \in 2^\N$, we define the Littlewood-Paley projection of $u$ by
\[
\widehat{u_N}(\xi) = \phi_N(\xi) \widehat{u}(\xi).
\]
With this notation, we can write
\[
u \sim \sum_{N \in 2^\N} u_N.
\]

Finally, let us recall the following useful fact:

\begin{proposition} \label{Double_MVT_Proposition}
Let $f \in C^2(\R^d)$ and $\xi, \eta, \lambda \in \R^d$ with $|\eta|,|\lambda| \ll |\xi|$. Then one has 
 \begin{equation*} 
|f(\xi+ \eta + \lambda) - f(\xi + \eta) - f(\xi +\lambda) + f(\xi)|  \lesssim |\eta|\, |\lambda| \, \| \nabla^2f(x)\|,
\end{equation*}   
where $|x| \sim |\xi|$, and $\| \cdot \|$ denotes a matrix norm on $d\times d$ matrices.
\end{proposition}
The proof of Proposition \ref{Double_MVT_Proposition} follows by applying the standard mean value theorem twice.

\section{Estimates in \texorpdfstring{$X^{s,b}$}{} spaces} \label{Estimates_Xsb} 

\subsection{Strichartz estimates} 

 In \cite{ChoOzawaXia}, Strichartz estimates were established for the general Schr\"{o}dinger equation 
 \[
 i u_t + \omega(|\nabla|) u = 0
 \]
 on $\R^d$, under some regularity assumptions on $\omega$. For $\omega(\rho)= -\rho^{2\alpha}$, in particular, the estimates obtained in \cite{ChoOzawaXia} imply the following Strichartz estimate for the fractional Laplacian. 
    \begin{lemma} \label{Lemma_fractional_Strichartz_estimate}
    Let $\alpha \in (\frac{1}{2},1]$ and let $(q,r)$ be a sharp Schr\"{o}dinger admissible pair. Then it holds that
    \begin{equation} \label{fractional_Strichartz_estimate}
     \| S_\alpha(t) u  \|_{L_t^qL_x^r} \lesssim \|u\|_{\dot{H}^s(\R^d)}, \qquad \text{ where } \quad   s := \frac{d}{2} - \pt{ \frac{2\alpha}{q} + \frac{d}{r}}.
    \end{equation}
    
    \end{lemma}
    Lemma \ref{Lemma_fractional_Strichartz_estimate} was also shown later in \cite{DinhPhD, Dinh_Applied_Mathematics}. Additionally, (\ref{fractional_Strichartz_estimate}) was used in \cite{HongSire} to obtain local well-posedness and ill-posedness for the fractional Schr\"{o}dinger equation in Sobolev spaces.
    
    By the transference principle (see  \cite[Lemma 2.9]{TaoDispersive}), (\ref{fractional_Strichartz_estimate}) implies the following bound in terms of the $X^{s,b}$-norm. 

\begin{corollary}\label{Strichartz_Xsb}
    Let $\alpha \in (\frac{1}{2},1]$, let $b>\frac{1}{2}$, and let $(q,r)$ be a sharp Schr\"{o}dinger admissible pair. Then it holds that
    \[
    \|u \|_{L_t^qL_x^r} \lesssim \|u\|_{X^{s,b}}, \qquad  \text{ where } \quad s := \frac{d}{2} - \pt{ \frac{2\alpha}{q} + \frac{d}{r}}.
    \]
\end{corollary}

 Further, by a standard Sobolev embedding, we can deduce the following.

\begin{corollary} \label{Strichartz_Xsb_embedding}
    Let $\alpha \in (\frac{1}{2},1]$, $d\geq 1$,  $b>\frac{1}{2}$,  and let $(q,r)$ be a Schr\"{o}dinger admissible pair. Then, for
    \[
     \begin{cases} s   
        \geq \frac{d}{2} - \pt{ \frac{2\alpha}{q} + \frac{d}{r}}, & 2\leq r<+\infty, \\
        s > \frac{d}{2} -  \frac{2\alpha}{q}, & r= + \infty,
    \end{cases} 
    \]
    we have
    \[
    \| u \|_{L_t^q L_x^r} \lesssim \|u\|_{X^{s,b}}.
    \]
\end{corollary}

\begin{proof}
 Assume for simplicity that $r < \infty $; the case $r=+\infty$ can be proved similarly. First, for $2\leq r_2  \leq r < + \infty$, by the Sobolev embedding theorem we have
    \[
    \| u\|_{L^{r}} \lesssim \|u\|_{W_x^{\sigma,r_2}}, \qquad \text{ for } \quad \sigma \geq d \pt{ \frac{1}{r_2} - \frac{1}{r} }.
    \]
    Then, for fixed $r,r_2$ and $q$ such that $2\leq r_2\leq r < + \infty$ and $\frac{2}{q}+ \frac{d}{r_2} = \frac{d}{2}$, it follows that
    \[
    \|u \|_{L_t^{q}L_x^{r}} \lesssim \| u\|_{L_t^q W_x^{\sigma, r_2}} =  \|  \langle \nabla \rangle^\sigma u\|_{L_t^q L_x^{ r_2}}\lesssim \|u\|_{X^{s+\sigma,b}}, \qquad \text{ for } \qquad  b> \frac{1}{2}, \enspace \text{ and } \enspace  s = \frac{d}{2} - \pt{ \frac{2\alpha}{q} + \frac{d}{r_2}},
    \]
    where the last inequality follows from Corollary \ref{Strichartz_Xsb} since $(q,r_2)$ is sharp Schr\"{o}dinger admissible. In particular, for $\sigma = d \pt{ \frac{1}{r_2} - \frac{1}{r} }$ we have
    \[
    s + \sigma =  \frac{d}{2} - \pt{ \frac{2\alpha}{q} + \frac{d}{r}},
    \]
    as desired. This completes the proof.
\end{proof}

\subsection{Linear and bilinear estimates in \texorpdfstring{$X^{s,b}$}{}}

By definition, we have
\begin{equation}\label{14.d.alpha}
\|u\|_{L^2_{t,x}} = \|u\|_{X^{0,0}}. 
\end{equation}

By Sobolev embedding in time, we have $X^{s,b} \hookrightarrow C_t H_x^s$ for $b> \frac{1}{2}$ and $s \in \R$, thus
\begin{equation}\label{20.d.alpha}
\|u\|_{L^\infty_t L^2_x} \lesssim \|u\|_{X^{0,\frac12+}}.
\end{equation}
In addition, since $H_x^s \hookrightarrow L_x^\infty$ for $s>\frac{d}{2}$, we also have
\begin{equation}\label{15.d.alpha}
\|u\|_{L^\infty_{t,x}} \lesssim \|u\|_{X^{\frac{d}{2}+,\frac{1}{2}+}}. 
\end{equation}

For $s_d(\alpha,p) := (\frac{d}{2}-\frac{2\alpha + d}{p})$, Corollary \ref{Strichartz_Xsb_embedding} gives 
\begin{equation}
    \|u\|_{L_{t,x}^{p}} \lesssim \|u\|_{X^{s_d(\alpha,p), \frac{1}{2}+}}
\end{equation}
for any $(p,p)$ Schr\"{o}dinger admissible with $p<+\infty$, in particular for  $p=p_d : = \frac{2(2+d)}{d}$.  Note that $s_d(\alpha, p_d) \geq 0$ for any $\alpha \leq 1$, and $p_d > 2$ for any $d \geq 1$. Thus, by interpolating with (\ref{14.d.alpha}) and (\ref{15.d.alpha}), we obtain
\[
    \|u\|_{L_{t,x}^{p_d-}} \lesssim \|u\|_{X^{s_d(\alpha, p_d), \frac{1}{2}-}}, \qquad  \|u\|_{L_{t,x}^{p_d+}} \lesssim \|u\|_{X^{s_d(\alpha, p_d)+, \frac{1}{2}+}}.
    \]
In particular, for $d=3$ we have
\begin{equation} \label{II_term_Y_bounds_plus_minus}
    \|u\|_{L_{t,x}^{\frac{10}{3}-}} \lesssim \| u\|_{X^{\frac{3}{5}(1-\alpha),\frac{1}{2}-}}, \qquad \|u\|_{L_{t,x}^{\frac{10}{3}+}} \lesssim \| u\|_{X^{\frac{3}{5}(1-\alpha)+,\frac{1}{2}+}}.    
    \end{equation}
    
Let $c_1,c_2 \in \R$ with $c_1<c_2$, and let $\chi = \chi(t)$ be the characteristic function of the interval $[c_1,c_2]$. Then, for any $s \in \R$ and $b< \frac{1}{2}$, 
\begin{equation} \label{chi_Xsb}
\|\chi f\|_{X^{s,b}}
\lesssim
\|f\|_{X^{s,b+}}.
\end{equation}
For the proof of (\ref{chi_Xsb}) in the periodic case, see \cite[Lemma 2.1]{Sohinger_S1}. The proof for the nonperiodic case is analogous.

Next, we prove a bilinear $X^{s,b}$-estimate, the proof of which is a slightly modified version of \cite[Theorem A.3.3]{DinhPhD}.

\begin{proposition}\label{Theorem_A.3.3_Dinh_modified}
Let $d\geq 2$, $\alpha \in (\frac{1}{2},1]$,  and $N_2, N_1\in 2^\N$ with $N_2 \leq N_1$.  Let $u,v \in L^2$ and suppose that
     \[
\operatorname{supp}\widehat{u}\subset\{\langle \xi \rangle \sim N_1\},\qquad
\operatorname{supp}\widehat{v}\subset\{\langle \xi \rangle \sim N_2\}.
\] 
Then it holds that
\begin{equation} \label{Fractional_Bilinear_estimate}
\| (S_\alpha(t) u ) (S_\alpha(t)v )  \|_{L_t^2 L_x^2} \lesssim  \frac{N_2^{\frac{d}{2}-\frac{1}{2}}}{N_1^{\alpha - \frac{1}{2}}} \; \Vert v \Vert_{L^2} \; \Vert u \Vert_{L^2},     
\end{equation}
where $S_\alpha$ is defined in (\ref{Salpha_definition}).
\end{proposition}

\begin{proof}
Consider first the case $N_1 \sim N_2$. By Hölder's inequality, we have
\begin{equation*}
\|(S_\alpha(t) u ) (S_\alpha(t)v )\|_{L^2_tL^2_x}
\;\leq\;
\|S_\alpha(t) v\|_{L^4_tL^4_x}\,\|S_\alpha(t) u\|_{L^4_tL^4_x}.
\end{equation*}

Since $(4,4)$ is Schr\"odinger admissible for any $d\geq 2$,  Corollary \ref{Strichartz_Xsb_embedding} gives
\[
\|S_\alpha(t)u\|_{L^4_tL^4_x} \lesssim \|u\|_{H^{\frac{d}{4}-\frac{\alpha}{2}}} \sim  N_1^{\frac{d}{4}-\frac{\alpha}{2}}\|u\|_{L^2},
\]
and similarly for $v$. Hence we obtain
\[
\| ( S_\alpha(t)u) (S_\alpha(t)v) \|_{L^2_tL^2_x}
\;\lesssim\;
N_2^{\frac{d}{4}-\frac{\alpha}{2}}  N_1^{\frac{d}{4}-\frac{\alpha}{2}}\,
\|v\|_{L^2}\,\|u\|_{L^2}.
\]
Since $N_2\sim N_1$, we have $ N_2^{\frac{d}{4}-\frac{\alpha}{2}}  N_1^{\frac{d}{4}-\frac{\alpha}{2}}  \sim N_2^{\frac{d}{2} - \frac{1}{2}} {N_1}^{-(\alpha - \frac{1}{2})}$. This gives (\ref{Fractional_Bilinear_estimate}) in the case $N_2 \sim N_1$. 

Let us now consider the case $N_2 \ll N_1$. By duality, it suffices to prove
\begin{equation}
  \Bigg| \int_{\R^d \times \R^d} G(|\xi|^{2\alpha} + |\eta|^{2\alpha}, \xi + \eta) \widehat{v}(\xi)\widehat{u}(\eta)  \; d\xi \, d\eta \Bigg| \lesssim \frac{N_2^{\frac{d}{2} - \frac{1}{2}}} {N_1^{\alpha  - \frac{1}{2}}} \; \| G\|_{L^2_{\tau,\xi}} \; \| \widehat{u} \|_{L^2_\xi} \; \| \widehat{v} \|_{L^2_\xi}, \label{Fractional_Bilinear_estimate_eq1}
\end{equation}
for all $G \in L^2_{\tau,\xi}$. We may assume without loss of generality that $ |\xi_1| \sim N_2$ and $ |\eta_1| \sim N_1$, where $\xi = (\xi_1, \underline{\xi})$ and $\eta = (\eta_1, \underline{\eta})$ with $\underline{\xi}, \underline{\eta} \in \R^{d-1}$. We make a change of variables $\tau = |\xi|^{2\alpha} + |\eta|^{2\alpha}$, $\zeta = \xi + \eta$ and $d\tau d\zeta = J d\xi_1 d\eta$, where
\[
J
= \Big|2\alpha|\xi|^{2\alpha-2}\xi_1 \pm 2\alpha|\eta|^{2\alpha-2}\eta_1\Big|.
\]

Since $\alpha > \frac{1}{2}$ and $N_2 \ll N_1$, we have $N_2^{2\alpha-1}\ll N_1^{2\alpha-1}$, and thus
\[
J\geq 2\alpha\Big(|\eta|^{2\alpha-2}|\eta_1|-|\xi|^{2\alpha-2}|\xi_1|\Big)
\gtrsim N_1^{2\alpha-1}-N_2^{2\alpha-1}\gtrsim N_1^{2\alpha-1}.
\]

Additionally, $ J\lesssim N_1^{2\alpha-1}+N_2^{2\alpha-1}\lesssim N_1^{2\alpha-1}$ trivially, thus $J \sim N_1^{2\alpha - 1}$. The rest of the proof is carried out as in \cite[Theorem A.3.3]{DinhPhD} using the change of variables on the LHS of (\ref{Fractional_Bilinear_estimate_eq1}) to deduce that
\[
\text{LHS of } (\ref{Fractional_Bilinear_estimate_eq1}) \lesssim \| G \|_{L^2_{\tau,\xi}}  \;\frac{N_2^{\frac{d}{2} - \frac{1}{2}}}{N_1^{\alpha - \frac{1}{2}}} \; \| \widehat{u} \|_{L^2_\xi} \; \| \widehat{v} \|_{L^2_\xi},
\]
for any $G \in L^2_{\tau,\xi}$. This completes the proof.
\end{proof}

 To translate Proposition \ref{Theorem_A.3.3_Dinh_modified} into a bilinear estimate in $X^{s,b}$ spaces, we use the following bilinear version of  \cite[Lemma 2.9]{TaoDispersive}.
 
\begin{lemma} \label{Lemma_A.2.7_DinhPhD}
Let  $s_1, s_2 \in \R$, and let $Y$ be a Banach space of functions on $\R \times \R^d$ such that
\[ 
\|(e^{it\tau} S_\alpha(t)f_1)(e^{it\zeta} S_\alpha(t) f_2)\|_Y \lesssim \|f_1\|_{H^{s_1}_x} \|f_2\|_{H^{s_2}_x}, 
\]
for all $f_1 \in H^{s_1}_x$, $f_2 \in H^{s_2}_x$ and all $\tau, \zeta \in \R$. Then it holds
\[ \|u_1u_2 \|_Y \lesssim \|u_1\|_{X^{s_1,\frac{1}{2}+}} \|u_2\|_{X^{s_2,\frac{1}{2}+}}, \]
for all $u_1 \in X^{s_1, \frac{1}{2}+}, u_2 \in X^{s_2, \frac{1}{2}+}$.
\end{lemma} 

The proof of Lemma \ref{Lemma_A.2.7_DinhPhD} is the same as \cite[Lemma A.2.7]{DinhPhD}, thus we omit it. The extension to all $u_1 \in X^{s_1, \frac{1}{2}+}, u_2 \in X^{s_2, \frac{1}{2}+}$ follows from the definition of the $X^{s,b}$ spaces in (\ref{Xsb_definition}).  

Proposition \ref{Theorem_A.3.3_Dinh_modified} and Lemma \ref{Lemma_A.2.7_DinhPhD}  then imply the following result.

\begin{corollary} \label{Proposition_2.3.d.alpha}
     Let $d\geq 2$, $\alpha \in (\frac{1}{2},1]$, and let $N_1, N_2 \in 2^\N$ with $N_2\leq N_1$. Let $u,v \in X^{0,\frac{1}{2}+}$ and suppose that
     \[
\operatorname{supp}\widehat{u}(t)\subset\{\langle \xi \rangle \sim N_1\},\qquad
\operatorname{supp}\widehat{v}(t)\subset\{\langle \xi \rangle \sim N_2\}.
\]
Then one has
     \begin{equation} \label{29.d.alpha}
         \left\|u v\right\|_{L_{t}^{2} L_{x}^{2}} \lesssim \frac{N_2^{\frac{d}{2}- \frac{1}{2}} } {N_1^{\alpha - \frac{1}{2}} }\left\|u\right\|_{X^{0, \frac{1}{2}+}}\left\|v\right\|_{X^{0, \frac{1}{2}+}}.  
     \end{equation}
\end{corollary}

In our analysis, we will have to work with a characteristic function $\chi$ of a time interval. To deal with $\chi$, we establish bilinear $X^{s,b}$-estimates where one of the components is multiplied by a characteristic function $\chi$. To obtain such bounds, we first need the following consequence of Corollary \ref{Proposition_2.3.d.alpha}.

\begin{corollary}\label{Corollary_2.4.d.alpha}
Let $d\geq 2$, $\alpha \in (\frac{1}{2},1]$, let $N_1, N_2 \in 2^\N$ with $N_2\leq N_1$.  Let $u,v \in X^{0,\frac{1}{2}+}$ and suppose that
     \[
\operatorname{supp}\widehat{u}(t)\subset\{\langle \xi \rangle \sim N_1\},\qquad
\operatorname{supp}\widehat{v}(t)\subset\{\langle \xi \rangle \sim N_2\}.
\]
Then it holds that
\begin{equation}\label{30.d2}
\|uv\|_{L^{2+}_t L^2_x} \;\lesssim\; \frac{N_2^{\frac{1}{2}}}{N_1^{(\alpha-\frac{1}{2})-}}
\|u\|_{X^{0,\frac12+}}\|v\|_{X^{0,\frac12+}}, \qquad d=2, 
\end{equation}
and
\begin{equation}\label{30.d3}
\|uv\|_{L^{2+}_t L^2_x} \;\lesssim\; \frac{N_2^{(\frac{d}{2} - \frac{1}{2})-}}{N_1^{(\alpha-\frac{1}{2})-}}
\|u\|_{X^{0,\frac12+}}\|v\|_{X^{0,\frac12+}}, \qquad d\geq 3. 
\end{equation}
\end{corollary}

\begin{proof}
By Bernstein's inequality and (\ref{20.d.alpha}), we observe that
\begin{equation}\label{31.d.alpha}
\|uv\|_{L^\infty_t L^2_x} \leq \|u\|_{L^\infty_t L^4_x}\|v\|_{L^\infty_t L^4_x}
\lesssim N_1^{\frac d4} N_2^{\frac d4}\,\|u\|_{L^\infty_t L^2_x}\|v\|_{L^\infty_t L^2_x}
\lesssim N_1^{\frac d4} N_2^{\frac d4}\,\|u\|_{X^{0,\frac12+}}\|v\|_{X^{0,\frac12+}}. 
\end{equation}
Then (\ref{30.d2}) and (\ref{30.d3}) follow by interpolating (\ref{29.d.alpha}) and (\ref{31.d.alpha}).  Note that, for $d=2$, the exponent of $N_2$ in both (\ref{29.d.alpha}) and (\ref{31.d.alpha}) is $\frac{1}{2}$, while for $d\geq 3$, the exponents of $N_2$ in (\ref{29.d.alpha}) and (\ref{31.d.alpha}) satisfy
\[
\frac{d}{4}<\frac{d}{2}-\frac{1}{2}.
\]
This completes the proof.
\end{proof}

Now, let $\chi(t)=\chi_{[t_0,t_0+\delta]}(t)$ be the characteristic function of the time interval $[t_0,t_0+\delta]$, where $\delta>0$ depends on the initial datum at time $t=t_0$.\footnote{This dependence is due to Proposition \ref{Proposition_3.1_d.alpha}, as we will see later on.} Since $\chi$ is not smooth, it is difficult to work with it directly. Instead, we decompose $\chi$ into two components that are easier to work with. This method was first used by  Colliander et al. \cite{CKSTT_global_not_refined} (see also \cite{Sohinger_Hartree_2D}).

First, fix  $\varphi\in C^\infty_0(\R)$ with $0\leq\varphi\leq1$ and $\int_{\R}\varphi(t)\,dt=1$.
For $\lambda>0$ define the normalised dilation $\varphi_\lambda(t):=\frac{1}{\lambda}\varphi\!\left(\frac{t}{\lambda}\right)$,
so that 
\[
\|\varphi_\lambda\|_{L^1_t}=\|\varphi\|_{L^1_t}=1.
\]

For any $N>1$  we decompose
\begin{equation} \label{equation_33.d.alpha}
\chi(t)=a(t)+b(t),\qquad a:=\chi*\varphi_{N^{-1}},\quad b:=\chi-a. 
\end{equation}

By  \cite[Lemma 8.2]{CKSTT_global_not_refined}, $a$ satisfies
\begin{equation} \label{34.d.alpha}
\|a(t) f\|_{X^{0,\frac12+}}\lesssim N^{0+}\|f\|_{X^{0,\frac12+}}.
\end{equation}

Moreover, for any $\ell\in(1,\infty)$, by Young's inequality we have
\begin{equation} \label{b_bound}
\|b\|_{L^\ell_t}=\|\chi-\chi*\varphi_{N^{-1}}\|_{L^\ell_t}\leq\|\chi\|_{L^\ell_t}+\|\chi*\varphi_{N^{-1}}\|_{L^\ell_t}
\leq 2\|\chi\|_{L^\ell_t}=C(\ell,\chi).    
\end{equation}

Define
\begin{equation} \label{b_1}
b_1(t):= \frac{1}{2\pi} \int_{\R}|\widehat b(\tau)|e^{it\tau}\, d\tau. 
\end{equation}
Then, by (\ref{b_bound}) it follows that\footnote{To obtain (\ref{36.d.alpha}) we argue as follows. By the definition of $b$ in  (\ref{equation_33.d.alpha}) and of $b_1$ in (\ref{b_1}), and the fact that 
\begin{equation*}
|\widehat{\chi}(\tau)| \lesssim_{\delta} \frac{1}{\langle \tau \rangle}\,,
\end{equation*}
it follows that 
\begin{equation}
\label{b_1_bound_1}
|\widehat{b_1}(\tau)|=|\widehat{b}(\tau)| \lesssim_{\delta} \frac{1}{\langle \tau \rangle}\,,
\end{equation}

From \eqref{b_1_bound_1}, it follows that 
\begin{equation}
\label{b_1_bound_2}
\|b_1\|_{H^{\frac{1}{2}-}_t} \leq C(\chi)\,.
\end{equation}

By Sobolev embedding and \eqref{b_1_bound_2}, we obtain, for $\ell$ large but finite, that 
\begin{equation*}
\label{b_1_bound_3}
\|b_1\|_{L^\ell_t} \lesssim_\ell \|b_1\|_{H^{\frac{1}{2}-}_t} \leq C(\ell,\chi)\,,
\end{equation*}
as desired.
}
\begin{equation}\label{36.d.alpha}
\|b_1\|_{L^\ell_t}\leq C(\ell,\chi)=C(\ell,u(t_0)), 
\end{equation}
where the final dependence follows because $\chi$ depends on $\delta$, and $\delta$ depends on the initial datum at time $t=t_0$.\\

Based on the analysis above, we can now prove the following useful result.

\begin{proposition} \label{Proposition_2.5.d.alpha}
Let $u,v\in X^{0,\frac12+}$ such that 
$\text{supp }\widehat{u}(t)\subset\{\langle\xi \rangle\sim N_1\},
\text{supp }\widehat{v}(t)\subset\{\langle \xi \rangle\sim N_2\}$, where $N_1 \gtrsim N_2$. Define
\[
\widetilde{u_1}(\tau,\xi):=\big|\widetilde{\chi u}(\tau,\xi)\big|,\qquad
\widetilde{\,v_1}(\tau, \xi):=\big|\widetilde{v}(\tau,\xi)\big|.
\]
Then we have
\begin{equation}\label{37.d.alpha}
\|u_1 v_1\|_{L^2_{t,x}}\;\lesssim \frac{N_2^{\frac{d}{2}-\frac{1}{2}}}{N_1^{(\alpha -\frac{1}{2})-}}
\|u\|_{X^{0,\frac12+}}\|v\|_{X^{0,\frac12+}}, \qquad d\geq 2.
\end{equation}
The same bounds hold if 
\begin{equation}  \label{Proposition_2.5.d.alpha_Second_case}
\widetilde{u_1}(\tau, \xi):=\big|\widetilde{ u}(\tau, \xi)\big|,\qquad
\widetilde{\, v_1}(\tau, \xi):=\big|\widetilde{\chi v}(\tau,\xi)\big|.    
\end{equation}
\end{proposition}

\begin{proof}
    First, note that, since $\chi(t) = a(t) + b(t)$, we have
    \[
    |\widetilde{\chi u}(\tau,\xi)| \leq \abs{\int_{\R} \widehat{a}(\sigma - \tau) \, \tilde{u}(\sigma, \xi) d \sigma } +  \abs{\int_{\R} \widehat{b}(\sigma - \tau) \, \tilde{u}(\sigma, \xi) d \sigma} = |\widetilde{a u}(\tau,\xi)| + |\widetilde{bu}(\tau,\xi)|. 
    \]
    Thus,
    \begin{align*}
        \|u_1 v_1\|_{L_{t,x}^2}  = \|\widetilde{u_1 v_1}\|_{L_{\tau,\xi}^2} & = \| |\widetilde{\chi u}| * |\tilde{v}| \|_{L_{\tau,\xi}^2}  \\[6pt]
        & \leq  \| |\widetilde{a u}| * |\widetilde{v}| \|_{L_{\tau,\xi}^2} +  \| |\widetilde{b u}| * |\widetilde{v}| \|_{L_{\tau,\xi}^2}   \\[6pt]
        & \sim \| u_{a}  v_1 \|_{L_{t,x}^2} +  \|u_{b}  v_1 \|_{L_{t,x}^2},
    \end{align*}
    where $\widetilde{u_{a}} := |\widetilde{au}|$ and $\widetilde{\, u_{b}} := |\widetilde{bu}|$. Notice that $\|u_{a}\|_{X^{0,\frac{1}{2}+}} = \|au\|_{X^{0,\frac{1}{2}+}}$,  and similarly $\|u_{b}\|_{X^{0,\frac{1}{2}+}} = \|bu\|_{X^{0,\frac{1}{2}+}}$, as well as $\|v_1\|_{X^{0,\frac{1}{2}+}} = \|v\|_{X^{0,\frac{1}{2}+}}$. Thus, by using (\ref{34.d.alpha}) with $N=N_1$ we get
    \[
    \|u_{a}\|_{X^{0,\frac{1}{2}+}} = \|au\|_{X^{0,\frac{1}{2}+}} \lesssim N_1^{0+} \|u\|_{X^{0,\frac{1}{2}+}}.
    \]
    
    Notice that $\operatorname{supp}{\widehat{u_{a}}}(t), \operatorname{supp}{\widehat{u_{b}}}(t) \subset \{\langle \xi \rangle \sim N_1\}$,  and $\operatorname{supp}{\widehat{v_{1}}}(t) \subset \{\langle \xi \rangle \sim N_2\}$. Thus, for $d=2$, 
     Corollary \ref{Proposition_2.3.d.alpha} implies that
        \begin{equation} \label{ua_bound}
         \| u_{a} \, v_1 \|_{L_{t,x}^2} \lesssim \frac{N_2^{\frac{1}{2}}}{N_1^{\alpha -\frac{1}{2}}}
\|u_{a}\|_{X^{0,\frac12+}}\|v_1\|_{X^{0,\frac12+}} \lesssim \frac{N_2^{\frac{1}{2}}}{N_1^{(\alpha -\frac{1}{2})-}}
\|u\|_{X^{0,\frac12+}}\|v\|_{X^{0,\frac12+}}.
        \end{equation}
        
        Now, $0 \leq \widetilde{{u}_{b}} \leq \widetilde{b_1 u_+}$, where $\widehat{b_1} = |\widehat{b}|$, and $\widetilde{u_+} = |\widetilde{u}|$. Thus, using (\ref{36.d.alpha}), for $\ell>2$ large we get
        \begin{equation} \label{ub_bound_1}
          \| u_{b} \, v_1 \|_{L_{t,x}^2} \leq   \|(b_1 u_+) \, v_1 \|_{L_{t,x}^2} \leq \|b_1\|_{L_t^\ell} \, \| u_+  v_1 \|_{L_t^{2+} L_x^2} \lesssim \| u_+  v_1 \|_{L_t^{2+} L_x^2}.
        \end{equation}
        
    Additionally, by Corollary \ref{Corollary_2.4.d.alpha}, and since $\|u_+\|_{X^{0,\frac{1}{2}+}} = \|u\|_{X^{0,\frac{1}{2}+}}$, we get
    \begin{equation} \label{ub_bound_2}
     \| u_+  v_1 \|_{L_t^{2+} L_x^2} \lesssim \frac{N_2^{\frac{1}{2}}}{N_1^{(\alpha - \frac{1}{2})-}}
\|u\|_{X^{0,\frac12+}}\|v\|_{X^{0,\frac12+}}.
    \end{equation}
    
    Combining the estimates (\ref{ua_bound}), (\ref{ub_bound_1}) and (\ref{ub_bound_2}), we deduce (\ref{37.d.alpha}).
   Likewise, for $d\geq 3$, we have
    \[
         \| u_{a}  v_1 \|_{L_{t,x}^2} \lesssim \frac{N_2^{\frac{d}{2}-\frac{1}{2}}}{N_1^{\alpha - \frac{1}{2}}}
\|u_{a}\|_{X^{0,\frac12+}}\|v_1\|_{X^{0,\frac12+}} \lesssim \frac{N_2^{\frac{d}{2}-\frac{1}{2}}}{N_1^{(\alpha -\frac{1}{2})-}}
\|u\|_{X^{0,\frac12+}}\|v\|_{X^{0,\frac12+}},
        \]
        and
        \[
     \| u_b  v_1 \|_{L_t^{2} L_x^2} \lesssim \frac{N_2^{(\frac{d}{2}-\frac{1}{2})- }}{N_1^{(\alpha -\frac{1}{2})-}}
\|u\|_{X^{0,\frac12+}}\|v\|_{X^{0,\frac12+}},
    \]
   which imply (\ref{37.d.alpha}). When $u_1$ and $v_1$ are given by (\ref{Proposition_2.5.d.alpha_Second_case}), we argue analogously. This completes the proof.
\end{proof}

\section{Proof of Theorem \texorpdfstring{\ref{Main_Theorem}}{} }  \label{Main_theorem_proof} 

\subsection{Definition of the \texorpdfstring{$\D$}{}-operator}

As mentioned in the introduction, our analysis is based on the \textit{upside-down I-method}. Thus, below we define an appropriate \textit{upside-down I-operator} \cite{Sohinger_S1, Sohinger_R, Sohinger_Hartree_2D}. First, let us define the following multiplier.

Fix $s > \frac{d}{2}$, and let $N > 1$ be given. We define $\theta = \theta_N : \R^d \rightarrow \R$ by
\[ 
\theta(\xi) := \begin{cases} 
\big(\frac{|\xi|}{N} \big)^s & \text{ if } |\xi| \geq 2N, \\
1 & \text{ if } |\xi| \leq N,
\end{cases} 
\] 
extended to all of $\R^d$ as a radial, smooth function that satisfies $0<c\leq  \theta(\xi) \leq C<\infty$ for $N\leq |\xi|\leq 2N$.  Arguing as in \cite[pages 19-20]{Sohinger_R}, one can show that 
\begin{equation} \label{theta_properties}
\|\nabla\theta(\xi)\| \lesssim \frac{\theta(\xi)}{|\xi|} , \qquad \|\nabla^2\theta(\xi)\| \lesssim \frac{\theta(\xi)}{|\xi|^2}.   
\end{equation}
Now, for $f : \R^d \rightarrow \mathbb{C}$, we define the operator $\D$ by
\[ 
\widehat{\D f}(\xi) := \theta(\xi) \widehat{f}(\xi).
\]
By the definition of $\D$, it holds that
\begin{equation} \label{D_scaling}
    \| \D f\|_{L^2} \lesssim \|f\|_{H^s} \lesssim N^s \| \D f\|_{L^2}.
\end{equation}

Moreover, we have the following useful inequality.

\begin{lemma} \label{Kato_Ponce_for_D}
    Let $f, g \in H^s(\R^d)$ with $s> \frac{d}{2}$. Then 
    \begin{equation*}
\|\mathcal D(fg)\|_{L^2}
\lesssim
\|f\|_{L^\infty}\|\mathcal D g\|_{L^2}
+
\|\mathcal D f\|_{L^2}\|g\|_{L^\infty},
    \end{equation*}
    where the implicit constant is independent of $N$.
\end{lemma}

The proof of Lemma \ref{Kato_Ponce_for_D} is an application of the Kato-Ponce inequality:

\begin{proposition}[Kato-Ponce inequality] \label{KP_extension_Hs_Linfty}
Let $\sigma\geq 0$. Then,
\begin{equation}\label{KP_Hs_Linfty}
\||\nabla|^\sigma(fg)\|_{L^2}
\lesssim 
\|f\|_{L^\infty}\||\nabla|^\sigma g\|_{L^2}
+
\||\nabla|^\sigma f\|_{L^2}\|g\|_{L^\infty},
\end{equation}
for all $f,g\in H^\sigma(\R^d)\cap L^\infty(\R^d)$.
\end{proposition}

We prove Proposition \ref{KP_extension_Hs_Linfty} in Section \ref{Appendix}.

Using Proposition \ref{KP_extension_Hs_Linfty}, we can then show Lemma \ref{Kato_Ponce_for_D}.

\begin{proof}[Proof of Lemma \ref{Kato_Ponce_for_D}]
    Let $f, g \in H^s(\R^d)$ with $s>\frac{d}{2}$. By analysing the regions $|\xi| \leq N$, $N \leq |\xi| \leq 2N$ and $|\xi| \geq 2N$, it is not difficult to see that
 \begin{equation} \label{Kato_Ponce_for_D_D_equivalence}
 \|\mathcal D f\|_{L^2}
\sim
\|f\|_{L^2}+N^{-s}\||\nabla|^s f\|_{L^2},    
 \end{equation}
where the implicit constant is independent of $N$. By H\"{o}lder's inequality, we have
\begin{equation}\label{Kato_Ponce_for_D_Holder}
\|fg\|_{L^2}
\leq
\|f\|_{L^\infty}\|g\|_{L^2}
+
\|g\|_{L^\infty}\|f\|_{L^2},
\end{equation}
while by the Kato-Ponce inequality (\ref{KP_Hs_Linfty}) we also have
\begin{equation}\label{Kato_Ponce_for_D_KatoPonce}
\||\nabla|^s(fg)\|_{L^2}
\lesssim
\|f\|_{L^\infty}\||\nabla|^s g\|_{L^2}
+
\||\nabla|^s f\|_{L^2}\|g\|_{L^\infty}.
\end{equation}
Combining (\ref{Kato_Ponce_for_D_D_equivalence}), (\ref{Kato_Ponce_for_D_Holder}) and (\ref{Kato_Ponce_for_D_KatoPonce}), we deduce the desired inequality.
\end{proof}

\subsection{Local-in-time bounds} Let $d \in \{2,3\}$ and let $u$ be the global solution of (\ref{equation_fractional_hartree}) on $\R^d$ with $s>\frac{d}{2}$. Our goal is to estimate $\| \D u(t)\|_{L^2}$, which in turn controls $\| u\|_{H^s}$ by (\ref{D_scaling}). To achieve this, we need the following local-in-time bounds.

\begin{proposition} \label{Proposition_3.1_d.alpha} 
Let $d \in \{2,3\}$ and let $u$ be the global solution of (\ref{equation_fractional_hartree}) on $\R^d$ for $s>\frac{d}{2}$ and $\alpha \in (\frac{d}{4},1]$.
Then, there exist $\delta=\delta\big(s,E(u_0),M(u_0)\big)$ and 
$C=C\big(s,E(u_0),M(u_0)\big)>0$, continuous in mass and energy, such that for every $t_0\in\R$ there exists a globally defined function $v:\R\times\R^d\to\mathbb{C}$ with
\begin{align}
& v|_{[t_0,t_0+\delta]} = u|_{[t_0,t_0+\delta]}, \label{46.d.alpha} \\[6pt]
& \|v\|_{X^{\alpha,\frac{1}{2}+}} \leq C\big(s,E(u_0),M(u_0)\big), \label{47.d.alpha} \\[6pt]
& \|\mathcal{D} v\|_{X^{0,\frac{1}{2}+}} \leq C\big(s,E(u_0),M(u_0)\big) \, \|\D u(t_0)\|_{L^2}. \label{48.d.alpha}
\end{align}
\end{proposition}

We present the proof of Proposition \ref{Proposition_3.1_d.alpha} in Section \ref{Appendix}.

As in \cite{Sohinger_S1, Sohinger_Hartree_2D}, we may assume that $u_0$ is smooth, as the following approximation lemma suggests.

\begin{lemma} \label{Proposition_3.2.d.alpha}
If $u$ satisfies
\[
\begin{cases}
i u_t - (-\Delta)^\alpha u = (V \ast |u|^2) u, \\
u(x, 0) = u_0(x),
\end{cases}
\]
and if the sequence $(u^{(n)})$ satisfies
\[
\begin{cases}
i u^{(n)}_t - (-\Delta)^\alpha u^{(n)} = (V \ast |u^{(n)}|^2) u^{(n)}, \\
u^{(n)}(x, 0) = u_0^{(n)}(x),
\end{cases} 
\]
where $u_0^{(n)} \in H^\infty(\R^d)$ and $u_0^{(n)} \xrightarrow{H^s} u_0$, then, for all $t$, 
\[
u^{(n)}(t) \xrightarrow{H^s} u(t).
\]
\end{lemma}

The proof of Lemma \ref{Proposition_3.2.d.alpha} follows from Proposition \ref{Proposition_3.1_d.alpha} and is similar to the proof of  \cite[Proposition 3.3]{Sohinger_S1}; we omit the details.

In order to argue by density, we need to show that the mass and energy are continuous on $H^s(\R^d)$. The former is clearly continuous on $H^\alpha(\R^d)$. For the latter, observe that by H\"{o}lder's and Young's inequalities, as well as a standard Sobolev embedding, we see that
\begin{align*}
    \abs{\int \big(V*(u_1u_2)\big)\,u_3\,u_4 d x } & \leq \| V * (u_1 u_2) \|_{L^2} \; \|u_3 u_4 \|_{L^2} \leq \|V\|_{L^1}\ \|u_1u_2\|_{L^2}   \; \|u_3 u_4 \|_{L^2} \\[6pt]
    & \leq \|V\|_{L^1}\ \|u_1\|_{L^4}\|u_2\|_{L^4}  \|u_3\|_{L^4}\|u_4\|_{L^4} \\[6pt]
    & \leq \|V\|_{L^1} \|u_1\|_{H^{\frac{d}{4}}}  \|u_2\|_{H^{\frac{d}{4}}} \|u_3\|_{H^{\frac{d}{4}}} \|u_4\|_{H^{\frac{d}{4}}}  \\[6pt]
    & \leq \|V\|_{L^1} \|u_1\|_{H^{\alpha}}  \|u_2\|_{H^{\alpha}} \|u_3\|_{H^{\alpha}} \|u_4\|_{H^{\alpha}}, 
\end{align*}
for $\alpha \geq \frac{d}{4}$.

Therefore, the energy is continuous on $H^\alpha$ and thus on $H^s$, since $s\geq \alpha$. Thus, if  $u_0^{(n)} \xrightarrow{H^s} u_0$ we deduce that
\begin{equation}\label{mass_energy_data_approximation}
     M(u_0^{(n)}) \ra M(u_0), \qquad E(u_0^{(n)}) \ra E(u_0) ,\qquad  \| u_0^{(n)}\|_{H^s} \ra \|u_0 \|_{H^s(\R^d) }.
\end{equation}

If Theorem \ref{Main_Theorem} were true for smooth solutions, it would then follow that, for all $t \in \R$, 
\[
\|u^{(n)}(t)\|_{H^s} \leq C\big(s, E(u_0^{(n)}), M(u_0^{(n)})\big) \; (1+|t|)^{\frac{1}{\kappa} s+} \|u_0^{(n)}\|_{H^s},
\]
where $\kappa = \frac{3}{2}(2\alpha - 1)$ for $d=2$, and $\kappa = (4\alpha-3)$ for $d=3$, as in Theorem \ref{Main_Theorem}. The general claim for non-smooth functions would then follow by an approximation argument using  (\ref{mass_energy_data_approximation}). We therefore assume from now on that $u_0 \in H^\infty(\R^d)$.

\subsection{A higher modified energy} \label{higher_modified_energy} 

As mentioned in the introduction, we use the method of \textit{higher modified energies}. We start by defining
\[
E_1(u(t)) := \|\D u(t)\|^2_{L^2}.
\]
A direct computation shows that
\[
\frac{d}{dt} E_1(u(t)) = ic \int_{\Gamma_4} \big[\theta^2(\xi_1) - \theta^2(\xi_2) + \theta^2(\xi_3) - \theta^2(\xi_4)\big] \, \widehat{V}(\xi_3 + \xi_4) \, \widehat{u}(\xi_1) \widehat{\overline u}(\xi_2) \widehat{u}(\xi_3)\widehat{\overline u}(\xi_4) \; d\Gamma_4(\xi),
\]
for some $c \in \C$. 

We then consider the \textit{higher modified energy}
\[
E_2(u) := E_1(u) + \lambda_4(M_4; u),
\]
 where $\lambda_4(M_4;u)$ is defined in (\ref{lambda_M_f}), and  $M_4 = M_4(\xi_1, \xi_2, \xi_3, \xi_4)$ is to be defined shortly (see (\ref{M4_definition}) below). Using (\ref{equation_fractional_hartree}) and (\ref{lambda_M_f}), a straightforward computation yields
 \begin{align*}
 \frac{d}{dt} \lambda_4(M_4; u) & = -i\lambda_4(M_4(\xi_1,\xi_2,\xi_3,\xi_4)\, (|\xi_1|^{2\alpha} - |\xi_2|^{2\alpha} + |\xi_3|^{2\alpha} - |\xi_4|^{2\alpha}); u)  \\
 & -i \int_{\Gamma_6} M_6(\xi_1, \xi_2, \xi_3, \xi_4, \xi_5, \xi_6)\, \widehat{u}(\xi_1) \widehat{\overline u}(\xi_2) \widehat{u}(\xi_3) \widehat{\overline u}(\xi_4) \widehat{u}(\xi_5)\widehat{\overline u}(\xi_6) \; d\Gamma_6(\xi),  
 \end{align*}
where $M_6$ is defined as
\begin{align} 
M_6(\xi_1, \xi_2, \xi_3, \xi_4, \xi_5, \xi_6) := & M_4(\xi_{123}, \xi_4, \xi_5, \xi_6)\widehat{V}(\xi_1 + \xi_2) - M_4(\xi_1, \xi_{234}, \xi_5, \xi_6)\widehat{V}(\xi_2 + \xi_3)  \notag \\ & + M_4(\xi_1, \xi_2, \xi_{345}, \xi_6)\widehat{V}(\xi_3 + \xi_4) - M_4(\xi_1, \xi_2, \xi_3, \xi_{456})\widehat{V}(\xi_4 + \xi_5), \label{M6_definition}
\end{align}
and $\xi_{ijk}$ is defined in (\ref{xi_ijpk}). Therefore,
\begin{align}
     & \frac{d}{dt} E_2(u)   =   \frac{d}{dt} E_1(u) + \frac{d}{dt} \lambda_4(M_4; u)  \notag\\[8pt]
    & =   ic \int_{\Gamma_4}  \big[ \big(\theta^2(\xi_1) - \theta^2(\xi_2) + \theta^2(\xi_3) - \theta^2(\xi_4) \big) \widehat{V}(\xi_3 + \xi_4) -  M_4 \; (|\xi_1|^{2\alpha} - |\xi_2|^{2\alpha} + |\xi_3|^{2\alpha} - |\xi_4|^{2\alpha}) \big]  \label{theta_M4_cancellation} \\[6pt]
    & \hspace{3cm}\widehat{u}(\xi_1) \widehat{\overline{u}}(\xi_2) \widehat{u}(\xi_3) \widehat{\overline{u}}(\xi_4) \; d\Gamma_4(\xi) \notag\\[8pt]
    & -  i\int_{\Gamma_6}   M_6(\xi_1,\xi_2,\xi_3,\xi_4,\xi_5,\xi_6) \; \widehat{u}(\xi_1) \widehat{\overline{u}}(\xi_2) \widehat{u}(\xi_3) \widehat{\overline{u}}(\xi_4) \widehat{u}(\xi_5) \widehat{\overline{u}}(\xi_6) \; d\Gamma_6(\xi).\notag
\end{align}
We wish to define $M_4$ appropriately, so that the quantity in the square brackets in (\ref{theta_M4_cancellation}) vanishes. A natural first attempt is to define $M_4$ by
\[
M_4 = c \frac{\theta^2(\xi_1)-\theta^2(\xi_2)+\theta^2(\xi_3)-\theta^2(\xi_4)}{|\xi_1|^{2\alpha}-|\xi_2|^{2\alpha}+|\xi_3|^{2\alpha}-|\xi_4|^{2\alpha}} \widehat{V}(\xi_3 + \xi_4), \qquad \text{provided } \quad |\xi_1|^{2\alpha}-|\xi_2|^{2\alpha}+|\xi_3|^{2\alpha}-|\xi_4|^{2\alpha} \neq 0. 
\]

However, in general, $\theta^2(\xi_1) - \theta^2(\xi_2) + \theta^2(\xi_3) - \theta^2(\xi_4)$ is not necessarily zero when $|\xi_1|^{2\alpha} - |\xi_2|^{2\alpha} + |\xi_3|^{2\alpha} - |\xi_4|^{2\alpha}$ is zero, thus there is no full cancellation in (\ref{theta_M4_cancellation}). Moreover, we do not have an appropriate way to bound the denominator of $M_4$ on the whole $\Gamma_4$, as we will see below (see Lemma \ref{Sufficient_and_necessary_nonperiodic}).

To resolve this, we argue as in \cite[pages 13-14]{Sohinger_Hartree_2D}. More precisely, we decompose 
\[
\Gamma_4 = \Omega_{nr} \cup \Omega_r,
\]
where $\Omega_{nr}$ is the set of nonresonant frequencies defined by
\[
\Omega_{nr} := \{(\xi_1, \xi_2, \xi_3, \xi_4) \in \Gamma_4: \; \xi_{12}, \xi_{14} \neq 0, |\cos\angle(\xi_{12}, \xi_{14})| > \gamma\},
\]
with $\gamma \in (0,1)$ to be chosen soon (see (\ref{angle_constraint}) below), and $\Omega_r := \Gamma_4 \setminus \Omega_{nr}$ is the set of resonant frequencies. Then, we set
\begin{equation} \label{M4_definition}
M_4(\xi_1, \xi_2, \xi_3, \xi_4) :=
\begin{cases}
    \displaystyle c \frac{\theta^2(\xi_1)-\theta^2(\xi_2)+\theta^2(\xi_3)-\theta^2(\xi_4)}{|\xi_1|^{2\alpha}-|\xi_2|^{2\alpha}+|\xi_3|^{2\alpha}-|\xi_4|^{2\alpha}} \widehat{V}(\xi_3 + \xi_4) & \text{if } (\xi_1, \xi_2, \xi_3, \xi_4) \in \Omega_{nr}, \\[4pt]
0 & \text{if } (\xi_1, \xi_2, \xi_3, \xi_4) \in \Omega_r.
\end{cases}
\end{equation}

By the definition of $M_4$, we observe that
\begin{align} 
 \frac{d}{dt} E_2(u) 
= & \int_{\Omega_{r}} \big[\theta^2(\xi_1)-\theta^2(\xi_2)+\theta^2(\xi_3)-\theta^2(\xi_4) \big] \; \widehat{V}(\xi_3+\xi_4)  \; \widehat{u}(\xi_1) \widehat{\overline u}(\xi_2) \widehat{u}(\xi_3) \widehat{\overline u}(\xi_4) \; d\Gamma_4(\xi)  \notag \\[8pt]
 & + \int_{\Gamma_6} M_6(\xi_1, \xi_2, \xi_3, \xi_4, \xi_5, \xi_6) \widehat{u}(\xi_1) \widehat{\overline u}(\xi_2) \widehat{u}(\xi_3) \widehat{\overline u}(\xi_4) \widehat{ u}(\xi_5) \widehat{\overline u}(\xi_6) \; d\Gamma_6(\xi). \label{E2_derivative}
\end{align}

Finally, to control $M_4$ on $\Omega_{nr}$, we use the following result.

\begin{lemma}[Resonance inequality] \label{Sufficient_and_necessary_nonperiodic}
Let $d\geq 2$, $\alpha\in\big(\tfrac12,1\big]$, and  $\gamma\in(0,1)$. Then
\begin{equation}  \label{angle_constraint}
 \gamma>\frac{1-\alpha}{\alpha}   
\end{equation} 
if and only if
\begin{empheq}[left=\empheqlbrace]{align}
 & \displaystyle \big|\,|\xi+\eta|^{2\alpha}-|\xi+\lambda+\eta|^{2\alpha}+|\xi+\lambda|^{2\alpha}-|\xi|^{2\alpha}\,\big| \gtrsim_{\alpha, \gamma}
\frac{|\eta|\,|\lambda|}{\big(|\xi|+|\eta|+|\lambda|\big)^{\,2-2\alpha}}, \label{important_inequality} \\[6pt]
& \text{ for all }  \xi,\eta,\lambda\in\R^d   \text{ with }   |\eta\cdot\lambda| \geq \gamma\,|\eta|\,|\lambda|. \label{vector_angle_constraint}
\end{empheq}
Moreover, the implicit constant in (\ref{important_inequality}) can be taken as 
\[
C(\alpha,\gamma) = 2\alpha (\alpha \gamma + \alpha - 1).
\]
\end{lemma} 
The proof of Lemma \ref{Sufficient_and_necessary_nonperiodic} is presented in Section \ref{Appendix}.

Note that (\ref{important_inequality}) is equivalent to 
\begin{equation} \label{important_inequality_equivalent_main_text}
\left| |\xi_1|^{2\alpha} - |\xi_2|^{2\alpha} + |\xi_3|^{2\alpha} - |\xi_{123}|^{2\alpha} \right| \gtrsim_{\alpha,\gamma} \frac{|\xi_1 + \xi_2|\; |\xi_2 + \xi_3|}{(|\xi_1| + |\xi_2| + |\xi_3|)^{2-2\alpha}}.  
\end{equation}
 Thus, as (\ref{angle_constraint}) suggests, we need to take
\[
\frac{1-\alpha}{\alpha} < \gamma < 1. 
\]
This will allow us to use (\ref{important_inequality}) and thus (\ref{important_inequality_equivalent_main_text}) to bound the denominator of $M_4$ on $\Omega_{nr}$
appropriately.

\subsection{A priori bounds for \texorpdfstring{$M_4$}{} and \texorpdfstring{$M_6$}{}}

In this section, we prove some estimates for the multipliers $M_4$ and $M_6$, defined in (\ref{M4_definition}) and (\ref{M6_definition}) respectively. First, given $(\xi_1, \xi_2, \xi_3, \xi_4) \in \Gamma_4$, we dyadically localise the frequencies as $\langle \xi_j \rangle \sim N_j$ where $N_j \in 2^\N$, for $1\leq j \leq 4$, and order the $N_j$'s such that $N^*_1 \geq N^*_2 \geq N^*_3 \geq N^*_4$. As in \cite{CKSTT_Resonant, Sohinger_S1, Sohinger_R,Sohinger_Hartree_2D}, we abuse notation and write $\theta(N^*_j)$ instead of $\theta(N^*_j, \vec{0})$. \\

First, we have the following bounds for $M_4$.

\begin{lemma} \label{Lemma_3.3}
With the dyadic localisation notation $ \langle \xi_j \rangle \sim N_j$ and the ordering
$N^*_1\geq N^*_2\geq N^*_3\geq N^*_4$, the multiplier $M_4$ satisfies
\begin{equation*}
|M_4| \lesssim_{\alpha,\gamma} \frac{1}{(N^*_1)^{2\alpha}}\,\theta(N^*_1)\,\theta(N^*_2).
\label{equation_M4_bound}
\end{equation*}
\end{lemma}

The proof of Lemma \ref{Lemma_3.3} is based on estimate (\ref{important_inequality_equivalent_main_text}) and is similar to \cite[Lemma 3.3]{Sohinger_Hartree_2D}. We omit it.

Next, we prove an estimate for $M_6$.

\begin{lemma} \label{lemma_for_sextilinear_term_B}
    Let $d\geq 2$, $s> \frac{d}{2}$, $\alpha \in (\frac{1}{2}, 1]$, $(\xi_1, \xi_2, \xi_3,\xi_4,\xi_5,\xi_6) \in \Gamma_6$, and assume that $\langle \xi_j\rangle \sim K_j$, for $1\leq j \leq 6$. Suppose also that we order the $K_j$'s such that $K_1^* \geq\ldots \geq K_6^*$. Then
    \[
    |M_6(\xi_1, \xi_2, \xi_3,\xi_4,\xi_5,\xi_6)|  \lesssim_{\alpha,\gamma}  \frac{1}{(K_1^*)^{2\alpha}} \theta(K_1^*) \theta(K_2^*).
    \]
\end{lemma}

\begin{proof}
Without loss of generality, we only estimate the first term appearing in the definition of $M_6$ in (\ref{M6_definition}); that is, we want to show that
    \begin{equation} \label{M6_to_M4_bound}
     |M_4(\xi_{123}, \xi_4, \xi_5, \xi_6)| \lesssim  \frac{1}{(K_1^*)^{2\alpha}} \theta(K_1^*) \theta(K_2^*).   
    \end{equation}
    
    If $(\xi_{123}, \xi_4, \xi_5, \xi_6) \in \Omega_{r}$, then $M_4 = 0$, thus we may assume that $(\xi_{123}, \xi_4, \xi_5, \xi_6) \in \Omega_{nr}$. By (\ref{important_inequality_equivalent_main_text})  and since $\widehat{V} \in L^\infty$, we have 
    \begin{align*}
       |M_4(\xi_{123}, \xi_4, \xi_5, \xi_6)| & \lesssim_{\alpha,\gamma} \frac{|\theta^2(\xi_{123}) - \theta^2(\xi_4) + \theta^2(\xi_5) - \theta^2(\xi_6)|}{|\xi_{1234}| \, |\xi_{1236}|} (|\xi_{123}| + |\xi_4|+|\xi_5|)^{2-2\alpha} \\[10pt]
       & \sim_{\alpha,\gamma}  \frac{|\theta^2(\xi_{456}) - \theta^2(\xi_4) + \theta^2(\xi_5) - \theta^2(\xi_6)|}{|\xi_{45}| \, |\xi_{56}|} (|\xi_{456}| + |\xi_4|+|\xi_5|)^{2-2\alpha},
    \end{align*}
    for all $(\xi_1, \xi_2, \xi_3,\xi_4,\xi_5,\xi_6) \in \Omega_{nr}$. Now, we  consider three main cases:
    \begin{enumerate}[itemsep=6pt,leftmargin=1.6cm]
        \item[\textbf{Case 1:}] $K_5 \geq K_6,K_4$, 
        \item[\textbf{Case 2:}] $K_4 \geq K_5,K_6$,
        \item[\textbf{Case 3:}] $K_6 \geq K_5,K_4$.
    \end{enumerate}

Cases $2$ and $3$ are symmetric, while Case $2$ is similar to Case $1$. Hence, we only consider Case 1, and in particular three subcases:
    \begin{enumerate}[leftmargin=2.2cm]  
        \item[\textbf{Subcase 1:}] $|\xi_5| \sim |\xi_{45}| \sim |\xi_{56}|$. In this case we have
        \[
        \frac{|\theta^2(\xi_{456}) - \theta^2(\xi_4) + \theta^2(\xi_5) - \theta^2(\xi_6)|}{|\xi_{45}| \, |\xi_{56}|} \lesssim \frac{\theta^2(\xi_5)}{|\xi_5|^2},
        \]
        hence
        \begin{equation} \label{Bound_Subcase_1}
         |M_4(\xi_{123}, \xi_4, \xi_5, \xi_6)| \lesssim  \frac{|\theta^2(\xi_{456}) - \theta^2(\xi_4) + \theta^2(\xi_5) - \theta^2(\xi_6)|}{|\xi_{45}| \, |\xi_{56}|} (|\xi_{456}| + |\xi_4|+|\xi_5|)^{2-2\alpha} \lesssim  \frac{\theta^2(\xi_5)}{|\xi_5|^{2\alpha}}.   
        \end{equation}
         \item[\textbf{Subcase 2:}] $|\xi_5| \sim |\xi_{56}| \gg |\xi_{45}|$ or $|\xi_5| \sim |\xi_{45}| \gg |\xi_{56}|$. By the mean value theorem and the first bound in (\ref{theta_properties}), we have
         \[
         |\theta^2(\xi_{456}) - \theta^2(\xi_6)| \lesssim |\xi_{45}| \, \frac{\theta^2(\xi_5)}{|\xi_5|}, \qquad \text{ and } \qquad |\theta^2(\xi_5)-\theta^2(\xi_4)| \lesssim |\xi_{45}| \, \frac{\theta^2(\xi_5)}{|\xi_5|},
         \]
         hence
         \[
          \frac{|\theta^2(\xi_{456}) - \theta^2(\xi_4) + \theta^2(\xi_5) - \theta^2(\xi_6)|}{|\xi_{45}| \, |\xi_{56}|} \lesssim \frac{|\xi_{45}|}{|\xi_{45}| \, |\xi_{56}|}\frac{\theta^2(\xi_5)}{|\xi_5|} \sim \frac{\theta^2(\xi_5)}{|\xi_5|^2},
         \]
         thus
         \begin{equation} \label{Bound_Subcase_2.1}
            |M_4(\xi_{123}, \xi_4, \xi_5, \xi_6)| \lesssim   \frac{\theta^2(\xi_5)}{|\xi_5|^{2\alpha}}. 
         \end{equation}
        The other case follows similarly; more specifically,
        \[
         |\theta^2(\xi_{456}) - \theta^2(\xi_4)| \lesssim |\xi_{56}| \, \frac{\theta^2(\xi_5)}{|\xi_5|}, \qquad \text{ and } \qquad |\theta^2(\xi_5)-\theta^2(\xi_6)| \lesssim |\xi_{56}| \, \frac{\theta^2(\xi_5)}{|\xi_5|},
         \]
         thus, as before,
         \[
          \frac{|\theta^2(\xi_{456}) - \theta^2(\xi_4) + \theta^2(\xi_5) - \theta^2(\xi_6)|}{|\xi_{45}| \, |\xi_{56}|} \lesssim \frac{|\xi_{56}|}{|\xi_{45}| \, |\xi_{56}|}\frac{\theta^2(\xi_5)}{|\xi_5|} \sim \frac{\theta^2(\xi_5)}{|\xi_5|^2},
         \]
         and hence
         \begin{equation}  \label{Bound_Subcase_2.2}
            |M_4(\xi_{123}, \xi_4, \xi_5, \xi_6)| \lesssim  \frac{\theta^2(\xi_5)}{|\xi_5|^{2\alpha}}. 
         \end{equation}
         \item[\textbf{Subcase 3:}] $|\xi_5| \gg |\xi_{56}|, |\xi_{45}|$. In this case, we use Proposition \ref{Double_MVT_Proposition}. More specifically, we write
         \begin{align*}
            \theta^2(\xi_5) - \theta^2(\xi_6) = \theta^2(\xi_5) - \theta^2(\xi_5 - \xi_{56}), \quad
             \theta^2(\xi_4) =\theta^2(-\xi_4) = \theta^2(\xi_5 - \xi_{45}),  \quad
              \theta^2(\xi_{456}) =  \theta^2(\xi_5-\xi_{45}-\xi_{56}),   
         \end{align*} 
    hence, by Proposition \ref{Double_MVT_Proposition} and the second bound in (\ref{theta_properties}), we get
         \begin{align*}
            |\theta^2(\xi_{456}) - \theta^2(\xi_4) + \theta^2(\xi_5) - \theta^2(\xi_6)| & = |\theta^2(\xi_5-\xi_{45}-\xi_{56}) - \theta^2(\xi_5 - \xi_{45}) + \theta^2(\xi_5) - \theta^2(\xi_5 - \xi_{56})|  \\[6pt]
           & \lesssim   |\xi_{45}| \, |\xi_{56}| \, \frac{\theta^2(\xi_5)}{|\xi_5|^2},
         \end{align*}
         and thus
         \begin{equation}  \label{Bound_Subcase_3}
            |M_4(\xi_{123}, \xi_4, \xi_5, \xi_6)| \lesssim  \frac{\theta^2(\xi_5)}{|\xi_5|^{2\alpha}}.
         \end{equation}
    \end{enumerate}
    
    To show (\ref{M6_to_M4_bound}), we first note that, if $\langle \xi_5 \rangle \ll N$, then
        \[
        \langle \xi_{123} \rangle = \langle \xi_4 + \xi_5 + \xi_6 \rangle \lesssim \langle \xi_5 \rangle \ll N,
        \]
        thus $|\xi_{123}|,|\xi_4|, |\xi_5|,| \xi_6| < N$, which implies $\theta^2(\xi_{123})-\theta^2(\xi_4)+\theta^2(\xi_5)-\theta^2(\xi_6) = 0$. Therefore, in this case, 
        \[
        M_4(\xi_{123}, \xi_4, \xi_5, \xi_6)= 0
        \]
        by definition, and hence (\ref{M6_to_M4_bound}) holds.
        
    Consequently, by (\ref{Bound_Subcase_1}), (\ref{Bound_Subcase_2.1}), (\ref{Bound_Subcase_2.2})  and (\ref{Bound_Subcase_3}), to show (\ref{M6_to_M4_bound}) it is enough to prove that
    \begin{equation} \label{Final_bound}
       \frac{\theta^2(\xi_5)}{|\xi_5|^{2\alpha}} \lesssim \frac{1}{(K_1^*)^{2\alpha}} \theta(K_1^*) \theta(K_2^*)
    \end{equation}
    for $ \langle \xi_5 \rangle \sim N$ and $ \langle \xi_5 \rangle \gg N$. To show (\ref{Final_bound}) for these two cases, we note that, since $(\xi_1, \ldots, \xi_6) \in \Gamma_6$, we always have that $K_1^* \sim K_2^*$, and thus $\theta(K_1^*) \sim \theta(K_2^*)$. Taking this into account, we can deduce (\ref{Final_bound}) as follows:
    \begin{itemize}
        \item $  \langle \xi_5 \rangle \sim N$. In this region, $\theta(\xi_5) \sim 1$ and thus
        \begin{equation*} \label{x5_bound}
         \frac{\theta(\xi_5)^2}{|\xi_5|^{2\alpha}} \sim \frac{1}{N^{2\alpha}}.   
        \end{equation*}
        
        If $K_1^* \sim N$, then $\theta(K_1^*) \sim 1$,  hence
        \[
        \frac{\theta(K_1^*)\theta(K_2^*)}{(K_1^*)^{2\alpha}} \sim \frac{1}{N^{2\alpha}}.
        \]
    
        If, on the other hand, $K_1^* \gg N$, then
        \[
        \frac{\theta(K_1^*)\theta(K_2^*)}{(K_1^*)^{2\alpha}} \sim \frac{(K_1^*)^{2s-2\alpha}}{N^{2s}} \gtrsim \frac{1}{N^{2\alpha}},
        \]
        where we used that $s\geq \alpha$ in the last inequality. Thus, in both cases, (\ref{Final_bound}) holds.
        \item $\langle \xi_5 \rangle \gg N$. In this case, we may assume that $|\xi_5|\geq 2N$. Since the map $r \mapsto \frac{\theta^2(r)}{r^{2\alpha}}$ is increasing on $[2N, + \infty)$ and $|\xi_5| \lesssim K_1^* \sim K_2^*$, (\ref{Final_bound}) follows.
    \end{itemize}
This completes the proof.
\end{proof}

\subsection{Proof of Theorem \ref{Main_Theorem}}

The proof of the main result follows by an iteration argument. First, we have the following result.

\begin{proposition} \label{Proposition_3.4_d.alpha}
Let $d \in \{2,3\}$ and let $u$ be the global solution of (\ref{equation_fractional_hartree}) for $s>\frac{d}{2}$ and $\alpha \in (\frac{d}{4},1]$. Then, for $N$ sufficiently large,
\begin{equation}
E_1(u(t)) \sim E_2(u(t))
\end{equation}
for every fixed time $t$, where the implicit constant is independent of $t$ and $N$.
\end{proposition}

\begin{proof}

Since  $E_2(u)-E_1(u)=\lambda_4(M_4;u)$, we have
\[
\big|\lambda_4(M_4;u)\big|
\leq
\int_{\Gamma_4}
|M_4(\xi_1,\xi_2,\xi_3,\xi_4)| \,
|\widehat{u}(\xi_1)|\,
|\widehat{\overline{u}}(\xi_2)|\,
|\widehat{u}(\xi_3)|\,
|\widehat{\overline{u}}(\xi_4)| \; d\Gamma_4(\xi).
\]

We dyadically localise the frequencies by  $\langle \xi_j \rangle \sim N_j$,
and order them such that $N_1^* \geq N_2^* \geq N_3^* \geq N_4^*$.
Notice that $N_1^* \sim N_2^*$ because $(\xi_1,\xi_2,\xi_3,\xi_4)\in \Gamma_4$. Moreover, the nonzero contributions occur when
\begin{equation} \label{Nj_conditions_Proposition_3.4_d.alpha}
\quad N_1^*\gtrsim N.    
\end{equation}
Let us denote the corresponding contribution to $\lambda_4(M_4;u)$ by $I_{N_1,N_2,N_3,N_4}$. Using Parseval's identity
and Lemma \ref{Lemma_3.3}, we obtain
\[
\big|I_{N_1,N_2,N_3,N_4}\big|
\lesssim
\frac{1}{(N_1^*)^{2\alpha}} \int_{\Gamma_4}
|\widehat{ \mathcal{D}u_{N_1^*}}(\xi_1)|
|\widehat{ \mathcal{D}\overline{u_{N_2^*}}}(\xi_2)|
|\widehat{u_{N_3^*}}(\xi_3)|
|\widehat{\overline{u_{N_4^*}}}(\xi_4)|\; d\Gamma_4(\xi).
\]
Here, we may assume without loss of generality that the complex conjugates are applied to the terms $u_{N_2^*}$ and $u_{N_4^*}$, as any other case is treated similarly. 

Now, for $j=1,\dots,4$, let us define the functions $F_1, \ldots, F_4$ on the space Fourier side by
\[
\widehat{F_1}:=\big|\widehat{ \mathcal{D}u_{N_1^*}}\big|,\qquad
\widehat{F_2}:=\big|\widehat{ \mathcal{D}\overline{ u_{N_2^*}}}\big|,\qquad
\widehat{F_3}:=\big|\widehat{u_{N_3^*}}\big|,\qquad
\widehat{F_4}:=\big|\widehat{\overline{ u_{N_4^*}}}\big|.
\]
By Parseval's identity, H\"{o}lder's inequality with exponents $L_x^2, L_x^2, L_x^\infty, L_x^\infty$, the Sobolev embedding $H_x^{\frac{d}{2}+} \hookrightarrow L_x^\infty$, the fact that $\|u\|_{H^\alpha} \lesssim 1$, and the fact that taking absolute values in the Fourier
transform does not change Sobolev norms, we obtain  
\begin{align*}
|I_{N_1,N_2,N_3,N_4}| & \lesssim \frac{1}{(N_1^*)^{2\alpha}}\int_{\R^d} F_1 \overline{F_2} F_3 \overline{F_4} \; dx \\[6pt] 
& \lesssim \frac{1}{(N_1^*)^{2\alpha}} \|F_1\|_{L_x^2}   \|F_2\|_{L_x^2} \,  \|F_3\|_{L^\infty_x}   \,   \|F_4\|_{L^\infty_x} \\[6pt]
& \lesssim  \frac{1}{(N_1^*)^{2\alpha}} \|F_1\|_{L_x^2}   \|F_2\|_{L_x^2} \,  \|F_3\|_{H_x^{\frac{d}{2}+}}   \,   \|F_4\|_{H_x^{\frac{d}{2}+}}  \\[6pt]
& \lesssim 
 \frac{1}{(N_1^*)^{2\alpha}} \; \|\mathcal D u_{N_1^*}\|_{L_x^2} \; \|\mathcal D u_{N_2^*}\|_{L_x^2}  \; (N_1^*)^{(\frac{d}{2}- \alpha)+} \; \|u_{N_3^*}\|_{H^\alpha}  \; (N_1^*)^{(\frac{d}{2} - \alpha) +} \|u_{N_4^*}\|_{H^\alpha} \\[6pt]
& \lesssim  \frac{1}{(N_1^*)^{(4\alpha-d)-}} \;  \|\mathcal D u\|_{L_x^2}^2 \; \|u\|_{H^\alpha}^2 \\[6pt]
& \lesssim  \frac{1}{(N_1^*)^{(4\alpha-d)-}}\;  E_1(u).
\end{align*}

Summing this over the $N_j$'s and using (\ref{Nj_conditions_Proposition_3.4_d.alpha}), we get
\[
|E_2(u) - E_1(u)| \lesssim  \frac{1}{N^{(4\alpha-d)-}} E_1(u).
\]
The exponent on $N$ is positive when $4\alpha-d >0$, that is when $\alpha > \frac{d}{4}$. This completes the proof.
\end{proof}

To argue via an iteration argument, we want to estimate
\[
E_2(u(t_0 + \delta)) - E_2(u(t_0)) = \int_{t_0}^{t_0+\delta} \frac{d}{dt} E_2(u(t))\;dt = \int_{t_0}^{t_0+\delta} \frac{d}{dt} E_2(v(t))\;dt,
\]
where $\delta>0$ and $v$ are as in Proposition \ref{Proposition_3.1_d.alpha}. We prove the following iteration bound.

\begin{lemma} \label{Lemma_3.5_fNLS}
    Let $d \in \{2,3\}$, let $\delta>0$ be as in Proposition \ref{Proposition_3.1_d.alpha}, and let $\alpha \in (\frac{d}{4},1]$. Then, for all $t_0 \in \R$, one has
    \[
    |E_2(u(t_0+\delta)) - E_2(u(t_0))| \lesssim \frac{1}{N^{\frac{3}{2}(2\alpha - 1)-}} E_2(u(t_0)), \qquad \text{ for } \enspace d=2, 
    \]
    and
    \[
    |E_2(u(t_0+\delta)) - E_2(u(t_0))| \lesssim \frac{1}{N^{(4\alpha-3)-}} E_2(u(t_0)), \qquad \text{ for } \enspace d=3.
    \]
\end{lemma}

\begin{proof}[Proof of Lemma \ref{Lemma_3.5_fNLS}]
First, without loss of generality, we may assume that $t_0 =0$. The general claim will follow by time translation
and the fact that all of the implied constants are uniform in time. By (\ref{E2_derivative}), we need to estimate
\begin{align*}
& \int_0^\delta 
    \int_{\Omega_r}
    \left( \theta^2(\xi_1) - \theta^2(\xi_2) + \theta(\xi_3)^2 - \theta^2(\xi_4) \right) \widehat{V}(\xi_3+\xi_4) \; \widehat{v}(\xi_1) \widehat{\overline{v}}(\xi_2) \widehat{v}(\xi_3) \widehat{\overline{v}}(\xi_4) \; d\Gamma_4(\xi) \, dt \\[6pt]
    & + \int_0^\delta \int_{\Gamma_6}
    M_6(\xi_1, \xi_2, \xi_3, \xi_4, \xi_5, \xi_6) \; \widehat{v}(\xi_1) \widehat{\overline{v}}(\xi_2) \widehat{v}(\xi_3) \widehat{\overline{v}}(\xi_4) \widehat{v}(\xi_5) \widehat{\overline{v}}(\xi_6) \; d\Gamma_6(\xi) \,dt \\[8pt]
& =: I + II,
\end{align*}
where $\delta>0$ and $v$ are as in Proposition \ref{Proposition_3.1_d.alpha}. We treat each term separately.

\noindent
\textbf{Estimate of $I$}: Without loss of generality, we may assume that 
\[
|\xi_1| \geq |\xi_2|,  |\xi_3|,  |\xi_4|, \enspace  \enspace |\xi_2| \geq |\xi_4|, \quad \text{ and  }  \quad |\xi_1| > N.
\]

Furthermore, we dyadically localise the frequencies, i.e.  $\langle \xi_j \rangle \sim N_j$,  $j = 1, \ldots, 4$, and order the $N_j$'s in decreasing order, that is $N_1^* \geq N_2^* \geq N_3^* \geq N_4^*$. Thus, we need to consider the contribution for which\footnote{Note that we may assume without loss of generality that $\xi_{12},\xi_{14} \neq 0$.}
\[
N_1^* \gtrsim N, \qquad |\cos \angle(\xi_{12},\xi_{14})| \leq \gamma <1.
\]
Since $\xi_1 + \xi_2 + \xi_3 + \xi_4 = 0$ and $|\cos \angle(\xi_{12},\xi_{14})| \leq \gamma <1$, we have
\begin{equation} \label{Nj_conditions}
 N_1 \sim N_2 \sim N_1^* \sim N_2^* \gtrsim N.
\end{equation}
To see the first relation, suppose for the sake of contradiction that $N_1 \gg N_2$. Then $N_1 \gg N_4$, and thus the vectors $\xi_{12}$ and $\xi_{14}$ form a very small angle. Hence, $|\cos \angle(\xi_{12},\xi_{14})|$ is close to $1$, which contradicts the assumption that $|\cos \angle(\xi_{12},\xi_{14})| \leq \gamma <1$.

Let us now consider the case $N_1^* \sim N_2^* \gg N_3^*,N_4^*$. Assuming this dyadic decomposition, we want to estimate
\begin{align} 
I_{N_1, N_2, N_3, N_4} : = 
\int_0^\delta & \int_{\Omega_r}  (\theta^2(\xi_1) - \theta^2(\xi_2) + \theta^2(\xi_3) - \theta^2(\xi_4)) \; \notag \\[6pt]
& \qquad \widehat{V}(\xi_3 + \xi_4)  \widehat{v_{N_1^*}}(\xi_1) \widehat{\overline{v_{N_2^*}}}(\xi_2) \widehat{v_{N_3^*}}(\xi_3) \widehat{\overline{v_{N_4^*}}}(\xi_4) \; d\Gamma_4(\xi) \,dt. \label{I_N1N2N3N4}
\end{align}
As before, we may assume without loss of generality that the complex conjugates are applied to the terms $v_{N_2^*}$ and $v_{N_4^*}$. 

Now, by (\ref{Nj_conditions}) and Proposition \ref{Double_MVT_Proposition},  arguing as in \cite[page 17-18]{Sohinger_Hartree_2D} yields
\begin{equation} \label{multiplier_double_mean}
    (\theta^2(\xi_1) - \theta^2(\xi_2) + \theta^2(\xi_3) - \theta^2(\xi_4)) \; \widehat{V}(\xi_3 + \xi_4) = O\left(   \frac{N^*_3}{N^*_1}\theta(N^*_1)\theta(N^*_2) \right).
\end{equation}

Consequently, by (\ref{multiplier_double_mean}) and Parseval's identity in time, we get
\begin{align*}
&|I_{N_1, N_2, N_3, N_4}|   \\
& \lesssim \int_{\tau_1+\tau_2+\tau_3+\tau_4=0} \int_{\Omega_r} \frac{N^*_3}{N^*_1} \theta(N^*_1)\theta(N^*_2) \\[6pt]
 &\qquad \qquad | \widetilde{v_{N_1^*}}(\tau_1,\xi_1)| |\widetilde{\chi \overline{ v_{N_2^*}}}( \tau_2,\xi_2)| |\widetilde{v_{N_3^*}}( \tau_3,\xi_3)| |\widetilde{\overline{v_{N_4^*}}}( \tau_4,\xi_4)| \; d\Gamma_4(\xi)\,d\Gamma_4(\tau)  \\[8pt]
& \lesssim 
\frac{1}{N_{1}^*}\!
\int_{\tau_1+\tau_2+\tau_3+\tau_4=0}
\int_{\Gamma_4}
\big| \reallywidetilde{\mathcal{D}v_{N_1^*}}(\tau_1,\xi_1)\big|
\big|\reallywidetilde{\chi \mathcal{D}\overline{v_{N_2^*}}}(\tau_2,\xi_2)\big|
\big|\reallywidetilde{\nabla v_{N_3^*}}(\tau_3, \xi_3)\big|
\big|\widetilde{\overline{v_{N_4^*}}}(\tau_4,\xi_4)\big| \;
d\Gamma_4(\xi)\,d\Gamma_4(\tau).
\end{align*}

Let us define the functions $F_1, \ldots, F_4$ on the space-time Fourier side by
\[
\widetilde{F}_1:=\big|\reallywidetilde{\mathcal{D}v_{N_1^*}}\big|,\qquad
\widetilde{F}_2:=\big|\reallywidetilde{\chi\mathcal{D}\overline{v_{N_2^*}}}\big|,\qquad
\widetilde{F}_3:=\big|\reallywidetilde{\nabla v_{N_3^*}}\big|,\qquad
\widetilde{F}_4:=\big|\widetilde{\overline{v_{N_4^*}}}\big|.
\]
By Parseval's identity in space and time we further have
\[
|I_{N_1,N_2,N_3,N_4}|
\;\lesssim\;
\frac{1}{N_{1}^*}\int_{\R}\int_{\R^d} F_1 \overline{F_2} F_3 \overline{F_4} \; dx \, dt.
\]

Consequently, by H\"{o}lder's inequality, Corollary \ref{Proposition_2.3.d.alpha} and Proposition \ref{Proposition_2.5.d.alpha} we obtain
\begin{align*}
&|I_{N_1,N_2,N_3,N_4}| \lesssim \frac{1}{N_1^*}\,\|F_1F_3\|_{L^2_{t,x}}\,\|F_2F_4\|_{L^2_{t,x}} \\[6pt]
&\lesssim \frac{1}{N_1^*}
\Big(\frac{(N_3^*)^{\frac{d}{2}-\frac{1}{2}}}{(N_1^*)^{\alpha -\frac{1}{2}}}\|\D v_{N_1^*}\|_{X^{0,\frac12+}}\|\nabla v_{N_3^*}\|_{X^{0,\frac12+}}\Big)
\Big(\frac{(N_4^*)^{\frac{d}{2}-\frac{1}{2}}}{(N_2^*)^{(\alpha -\frac{1}{2})-}}\|\D v_{N_2^*}\|_{X^{0,\frac12+}}\|v_{N_4^*}\|_{X^{0,\frac12+}}\Big) \\[6pt]
&\sim \frac{1}{N_1^*}
\Big(\frac{(N_3^*)^{\frac{d}{2}-\frac{1}{2}}}{(N_1^*)^{\alpha -\frac{1}{2}}}\|\D v_{N_1^*}\|_{X^{0,\frac12+}} (N_3^*)^{1-\alpha} \|v_{N_3^*}\|_{X^{\alpha,\frac12+}}\Big)
\Big(\frac{(N_4^*)^{\frac{d}{2}-\frac{1}{2}}}{(N_2^*)^{(\alpha -\frac{1}{2})-}}\|\D v_{N_2^*}\|_{X^{0,\frac12+}}  (N_4^*)^{-\alpha} \|v_{N_4^*}\|_{X^{\alpha,\frac12+}}\Big).
\end{align*}
For $d=2$, by Propositions \ref{Proposition_3.1_d.alpha} and \ref{Proposition_3.4_d.alpha} we further get
\begin{align*}
|I_{N_1,N_2,N_3,N_4}|
&\lesssim \frac{1}{N_1^*}
\Big(\frac{(N_1^*)^{\frac{3}{2}-\alpha}}{(N_1^*)^{\alpha -\frac{1}{2}}}\|\D v_{N_1^*}\|_{X^{0,\frac12+}}  \|v_{N_3^*}\|_{X^{\alpha,\frac12+}}\Big)
\Big(\frac{1}{(N_1^*)^{(\alpha -\frac{1}{2})-}}\|\D v_{N_2^*}\|_{X^{0,\frac12+}}   \|v_{N_4^*}\|_{X^{\alpha,\frac12+}}\Big) \\[6pt]
& \lesssim \frac{1}{(N_1^*)^{(3\alpha-\frac{3}{2})-}} E_2(u_0).
\end{align*}

Summing over all $N_j$ and using (\ref{Nj_conditions}), we deduce that
\begin{equation} \label{I_2D}
    |I| \lesssim \frac{1}{N^{(3\alpha-\frac{3}{2})-}} E_2(u_0).
\end{equation} 
Thus, we require $3\alpha-\frac{3}{2} >0$, that is $\alpha > \frac{1}{2}$. \\

Similarly, for $d=3$ we have
\begin{align*}
|I_{N_1,N_2,N_3,N_4}|
&\lesssim \frac{1}{N_1^*}
\Big(\frac{(N_1^*)^{2-\alpha}}{(N_1^*)^{\alpha -\frac{1}{2}}}\|\D v_{N_1^*}\|_{X^{0,\frac12+}}  \|v_{N_3^*}\|_{X^{\alpha,\frac12+}}\Big)
\Big(\frac{(N_4^*)^{1-\alpha}}{(N_2^*)^{(\alpha -\frac{1}{2})-}}\|\D v_{N_2^*}\|_{X^{0,\frac12+}}   \|v_{N_4^*}\|_{X^{\alpha,\frac12+}}\Big) \\[6pt]
& \lesssim \frac{1}{(N_1^*)^{(4\alpha - 3) -}} E_2(u_0),
\end{align*}
which then implies
\begin{equation} \label{I_3D}
    |I| \lesssim \frac{1}{N^{(4\alpha - 3) -}} E_2(u_0).
\end{equation}
Thus, we need $4\alpha - 3  >0$, i.e. $\alpha > \frac{3}{4}$. \\

\noindent
\textbf{Estimate of $II$}: Arguing as above, we dyadically localise $\xi_1,\xi_2,\ldots,\xi_6$, i.e. $\langle \xi_j \rangle \sim K_j$ and we order them such that $K_1^* \geq K_2^* \geq \ldots \geq K_6^*$. Hence, we have $K_1^* \sim K_2^*$, and the nonzero contributions occur when
\begin{equation} \label{Kj_conditions}
K_1^*  \gtrsim N.    
\end{equation}

Now, let us consider the case $K_1^* \sim K_2^* \gg K_3^*,K_4^*,K_5^*,K_6^*$. Thus, we need to estimate the corresponding contribution to $II$, defined by
\begin{align*}
II_{K_1,K_2,K_3,K_4,K_5,K_6}  &:= \int_0^\delta \int_{\Gamma_6}
M_6(\xi_1,\xi_2,\xi_3,\xi_4,\xi_5,\xi_6)\,\widehat{v_{K_1^*}}(\xi_1) \widehat{\overline{v_{K_2^*}}}(\xi_2) \\[6pt]
& \hspace{5cm} \widehat{v_{K_3^*}}(\xi_3)\,\widehat{\overline{v_{K_4^*}}}(\xi_4)\,
 \widehat{v_{K_5^*}}(\xi_5)\,\widehat{\overline{v_{K_6^*}}}(\xi_6)\; d\Gamma_6(\xi)\,dt.
\end{align*}
As before, we may assume without loss of generality that the complex conjugates are applied to the terms $v_{K_2^*}$, $v_{K_4^*}$ and $v_{K_6^*}$. 

By Parseval's identity and Lemma \ref{lemma_for_sextilinear_term_B}, we have
\begin{align}
 & |II_{K_1,K_2,K_3,K_4,K_5,K_6}|  \notag \\[6pt]
 & \lesssim   \frac{\theta(K_1^*)\theta(K_2^*)}{(K_1^*)^{2\alpha}}  \int_{\tau_1 + \ldots + \tau_6=0}  \int_{\Gamma_6}
 |\widetilde{ \chi v_{K_1^*}}(\tau_1,\xi_1)| \, |\widetilde{\overline{v_{K_2^*}}}(\tau_2,\xi_2)| \notag \\[6pt] 
& \hspace{4.8cm} |\widetilde{ v_{K_3^*}}(\tau_3,\xi_3)| \,|\widetilde{\overline{v_{K_4^*}}}(\tau_4,\xi_4)| \,
|\widetilde{ v_{K_5^*}}(\tau_5,\xi_5)| \,|\widetilde{\overline{v_{K_6^*}}}(\tau_6,\xi_6)| \; d\Gamma_6(\xi)\, d\Gamma_6(\tau)  \notag \\[10pt]
 & \lesssim \frac{1}{ (K_1^*)^{2\alpha}}   \int_{\tau_1 + \ldots + \tau_6=0}  \int_{\Gamma_6}
 \,|\reallywidetilde{\chi \D v_{K_1^*}}(\tau_1,\xi_1)| \,
|\reallywidetilde{  \D \overline{v_{K_2^*}}}(\tau_2,\xi_2)|  \notag \\[6pt]
& \hspace{5cm} |\widetilde{v_{K_3^*}}(\tau_3,\xi_3)|
\, |\widetilde{\overline{v_{K_4^*}}}(\tau_4,\xi_4)|  
\, |\widetilde{v_{K_5^*}}(\tau_5,\xi_5)| 
\, |\widetilde{\overline{v_{K_6^*}}}(\tau_6,\xi_6)|  \; d\Gamma_6(\xi)\, d\Gamma_6(\tau).   \label{II_example_term}
\end{align}

Now let us define the functions $F_1, \ldots, F_6$ on the space-time Fourier side by
\[ 
\widetilde{F_1}:=|\reallywidetilde{\chi \D v_{K_1^*}}|,\qquad 
\widetilde{F_2}:=|\reallywidetilde{  \D \overline{v_{K_2^*}}}|,\qquad
\widetilde{F_3}:=|\widetilde{ v_{K_3^*}}|,\qquad
\widetilde{F_4}:=|\widetilde{ \overline{v_{K_4^*}}}|,  \qquad
\widetilde{F_5}:=|\widetilde{ v_{K_5^*}}|, \qquad
\widetilde{F_6}:=|\widetilde{ \overline{v_{K_6^*}}}|.
\]
Then, by Parseval's identity in space and time, we have
\begin{equation} \label{II_term_Parseval}
|II_{K_1,K_2,K_3,K_4,K_5,K_6}|\lesssim \frac{1}{(K^*_1)^{2\alpha}}\int_{\R}\int_{ \R^d} F_1 \overline{F_2} F_3 \overline{F_4} F_5 \overline{F_6} \,dx\,dt.    
\end{equation}

 To proceed, we consider the two cases, $d=2$ and $d=3$, separately.
 
 First, for $d=2$, we use H\"{o}lder's inequality with exponents $(2, 8, 8,8,8)$ to deduce that (\ref{II_term_Parseval}) is
\begin{equation} \label{II_term_Holder_1}
\lesssim  \frac{1}{(K^*_1)^{2\alpha}} \|F_1 F_3\|_{L_{t,x}^{2}} \, \|F_2\|_{L_{t,x}^{8}} \,  \|F_4\|_{L_{t,x}^8} \, \|F_5\|_{L_{t,x}^8} \, \|F_6\|_{L_{t,x}^8}.    
\end{equation}
Then, by Propositions \ref{Proposition_2.5.d.alpha},  \ref{Proposition_3.1_d.alpha} and \ref{Proposition_3.4_d.alpha} and Corollary \ref{Strichartz_Xsb_embedding},  (\ref{II_term_Holder_1}) is
\begin{align*}
& \lesssim  \frac{1}{(K_1^*)^{2\alpha }}\frac{1}{(K_1^*)^{(\alpha -\frac{1}{2})-}}
\|\D v\|_{X^{0,\frac12+}}\|v\|_{X^{\alpha,\frac12+}} (K_2^*)^{\frac{3}{4}-\frac{\alpha}{4}}  \|\mathcal{D}v\|_{X^{0,\frac{1}{2}+}} (K_1^*)^{3(\frac{3}{4}-\frac{\alpha}{4} - \frac{1}{2})}  \| v\|_{X^{\alpha,\frac{1}{2}+}}^3 \\[6pt]
& \lesssim  \frac{1}{(K_1^*)^{(4\alpha - 2)-}}
\|\D v\|_{X^{0,\frac12+}}^2   \, \|v\|_{X^{\alpha,\frac12+}}^4 \\[6pt]
& \lesssim  \frac{1}{(K_1^*)^{(4\alpha - 2)-}} E_2(u_0),
\end{align*}
and by summing over all $K_j$'s and using (\ref{Kj_conditions}) we conclude that
\begin{equation} \label{II_2D}
    |II| \lesssim  \frac{1}{N^{(4\alpha - 2)-}} E_2(u_0).
\end{equation}
Thus we need $4\alpha -  2 >0$, that is $ \alpha > \frac{1}{2} $. 

For the case $d=3$, H\"{o}lder's inequality with exponents $(\frac{10}{3}-, \frac{10}{3}+, 10,10,10,10)$ implies that (\ref{II_term_Parseval}) is
\begin{equation} \label{II_term_Holder_2}
   \lesssim  \frac{1}{(K^*_1)^{2\alpha}} \|F_1\|_{L_{t,x}^{\frac{10}{3}-}} \|F_2\|_{L_{t,x}^{\frac{10}{3}+}} \, \|F_3\|_{L_{t,x}^{10}} \,  \|F_4\|_{L_{t,x}^{10}} \, \|F_5\|_{L_{t,x}^{10}} \, \|F_6\|_{L_{t,x}^{10}}.  
\end{equation}
Therefore, by Propositions  \ref{Proposition_3.1_d.alpha} and \ref{Proposition_3.4_d.alpha}, Corollary \ref{Strichartz_Xsb_embedding}, (\ref{II_term_Y_bounds_plus_minus})  and (\ref{chi_Xsb}), (\ref{II_term_Holder_2}) is 
\begin{align*}
& \lesssim  \frac{1}{(K_1^*)^{2\alpha }} (K_1^*)^{\frac{6}{5}(1-\alpha)+} \|\mathcal D v\|_{X^{0,\frac{1}{2}+}}^2  (K_1^*)^{4(\frac{12}{10}- \frac{2\alpha}{10}-\alpha )}  \| v\|_{X^{\alpha,\frac{1}{2}+}}^4 \\[6pt]
& \lesssim  \frac{1}{(K_1^*)^{(8\alpha-6)- }}  \|\mathcal D v\|_{X^{0,\frac{1}{2}+}}^2   \| v\|_{X^{\alpha,\frac{1}{2}+}}^4 \\[6pt]
& \lesssim  \frac{1}{(K_1^*)^{(8\alpha-6)- }} E_2(u_0),
\end{align*}
and, as before, by summing over all $K_j$'s and using (\ref{Kj_conditions}) we conclude that
\begin{equation} \label{II_3D}
    |II| \lesssim   \frac{1}{N^{(8\alpha-6)- }} E_2(u_0).
\end{equation}
Thus we need $8\alpha-6 >0 $, i.e. $\alpha > \frac{3}{4}$. 

We note that the above analysis only considers the cases $N_1^* \sim N_2^* \gg N_3^*, N_4^*$ and $K_1^* \sim K_2^* \gg K_3^*, K_4^*, K_5^*, K_6^*$, because every other case gives the same or a better bound. To see this, let us consider the term $I_{N_1,N_2,N_3,N_4}$ defined in (\ref{I_N1N2N3N4}) with $N_1^* \sim N_2^* \geq N_3^*, N_4^*$. Note that since $(\xi_1,\xi_2,\xi_3,\xi_4) \in \Gamma_4$, we always have $N_1^* \sim N_2^*$, thus there are three cases to consider:
\[
N_1^* \sim N_2^* \sim N_3^* \sim N_4^*, \qquad N_1^* \sim N_2^* \sim N_3^* \gg N_4^* \qquad \text{ and} \qquad N_1^* \sim N_2^* \gg N_3^*, N_4^*.
\]

For any of the above cases, we have
\begin{align}
  & |I_{N_1,N_2,N_3,N_4}|  \notag \\[6pt]
  & \lesssim  \frac{1}{N_1^*}
\Big(\frac{(N_3^*)^{\frac{d}{2}-\frac{1}{2}}}{(N_1^*)^{\alpha -\frac{1}{2}}}\|\D v_{N_1^*}\|_{X^{0,\frac12+}}\|\nabla v_{N_3^*}\|_{X^{0,\frac12+}}\Big)
\Big(\frac{(N_4^*)^{\frac{d}{2}-\frac{1}{2}}}{(N_2^*)^{(\alpha -\frac{1}{2})-}}\|\D v_{N_2^*}\|_{X^{0,\frac12+}}\|v_{N_4^*}\|_{X^{0,\frac12+}}\Big) \notag \\[6pt]
&\sim \frac{1}{N_1^*}
\Big(\frac{(N_3^*)^{\frac{d}{2}-\frac{1}{2}}}{(N_1^*)^{\alpha -\frac{1}{2}}}\|\D v_{N_1^*}\|_{X^{0,\frac12+}} (N_3^*)^{1-\alpha} \|v_{N_3^*}\|_{X^{\alpha,\frac12+}}\Big)
\Big(\frac{(N_4^*)^{\frac{d}{2}-\frac{1}{2}}}{(N_2^*)^{(\alpha -\frac{1}{2})-}}\|\D v_{N_2^*}\|_{X^{0,\frac12+}}  (N_4^*)^{-\alpha} \|v_{N_4^*}\|_{X^{\alpha,\frac12+}}\Big)  \notag \\[6pt]
& =   \frac{1}{N_1^*} \;
\frac{\|\D v_{N_1^*}\|_{X^{0,\frac12+}}  \|\D v_{N_2^*}\|_{X^{0,\frac12+}} \|v_{N_3^*}\|_{X^{\alpha,\frac12+}}
   \|v_{N_4^*}\|_{X^{\alpha,\frac12+}}}{(N_1^*)^{\alpha -\frac{1}{2}} (N_2^*)^{(\alpha -\frac{1}{2})-} (N_3^*)^{-(\frac{d}{2}-\frac{1}{2} + 1 - \alpha)}  (N_4^*)^{-(\frac{d}{2}-\frac{1}{2} - \alpha)}}    
  \notag \\[6pt]
& =   \frac{1}{(N_1^*)^{f_0(\alpha,d)}} \;
\frac{\|\D v_{N_1^*}\|_{X^{0,\frac12+}}  \|\D v_{N_2^*}\|_{X^{0,\frac12+}} \|v_{N_3^*}\|_{X^{\alpha,\frac12+}}
   \|v_{N_4^*}\|_{X^{\alpha,\frac12+}}}{(N_1^*)^{f_1(\alpha,d)} (N_2^*)^{f_2(\alpha,d)} (N_3^*)^{f_3(\alpha,d)}  (N_4^*)^{f_4(\alpha,d)}}   ,  \label{I_last}
\end{align}
where
\[
f_0(\alpha,d) := 1, \enspace  f_1(\alpha,d) := \alpha -\frac{1}{2}, \enspace  f_2(\alpha,d) := (\alpha -\frac{1}{2})-, \enspace  f_3(\alpha,d) := -(\frac{d}{2}-\frac{1}{2} + 1 - \alpha), \enspace  f_4(\alpha,d) := -(\frac{d}{2}-\frac{1}{2} - \alpha).
\]

Now, suppose first that  $N_1^* \sim N_2^* \sim N_3^* \sim N_4^*$. Then, (\ref{I_last}) is
\begin{align*}
& \sim   \frac{1}{(N_1^*)^{f_0(\alpha,d)}} \;
\frac{\|\D v_{N_1^*}\|_{X^{0,\frac12+}}  \|\D v_{N_2^*}\|_{X^{0,\frac12+}} \|v_{N_3^*}\|_{X^{\alpha,\frac12+}}
   \|v_{N_4^*}\|_{X^{\alpha,\frac12+}}}{(N_1^*)^{f_1(\alpha,d)} (N_1^*)^{f_2(\alpha,d)} (N_1^*)^{f_3(\alpha,d)}  (N_1^*)^{f_4(\alpha,d)}}     \\[6pt]
& \sim \frac{1}{(N_1^*)^{L_4}} \;  \|\D v_{N_1^*}\|_{X^{0,\frac12+}}  \|\D v_{N_2^*}\|_{X^{0,\frac12+}} \|v_{N_3^*}\|_{X^{\alpha,\frac12+}}
   \|v_{N_4^*}\|_{X^{\alpha,\frac12+}},
\end{align*}
where
\[
L_4 = L_4(\alpha,d) := f_0(\alpha,d) + f_1(\alpha,d) + f_2(\alpha,d) + f_3(\alpha,d) + f_4(\alpha,d)  = (4\alpha - d)-.
\]

Summing over all $N_j$ and using (\ref{Nj_conditions}) yields
\begin{equation*} 
    |I| \lesssim \frac{1}{N^{L_4 -}} E_2(u_0)  = \frac{1}{N^{(4\alpha - d)-}} E_2(u_0),
\end{equation*}
assuming $L_4- >0$.

If however $N_1^* \sim N_2^* \sim N_3^* \gg N_4^*$, then $(N_4^*)^{-f_4} \lesssim  (N_1^*)^{\max (0,-f_4)} =(N_1^*)^{ -\min (0,f_4)}$. Therefore, in this case, (\ref{I_last}) is
\begin{align*}
& \lesssim  \frac{1}{(N_1^*)^{f_0(\alpha,d)}} \;
\frac{\|\D v_{N_1^*}\|_{X^{0,\frac12+}}  \|\D v_{N_2^*}\|_{X^{0,\frac12+}} \|v_{N_3^*}\|_{X^{\alpha,\frac12+}}
   \|v_{N_4^*}\|_{X^{\alpha,\frac12+}}}{(N_1^*)^{f_1(\alpha,d)} (N_1^*)^{f_2(\alpha,d)} (N_1^*)^{f_3(\alpha,d)}  (N_1^*)^{\min(0,f_4(\alpha,d))}}    \\[6pt]
& \sim \frac{1}{(N_1^*)^{L_3}} \|\D v_{N_1^*}\|_{X^{0,\frac12+}}  \|\D v_{N_2^*}\|_{X^{0,\frac12+}} \|v_{N_3^*}\|_{X^{\alpha,\frac12+}}
   \|v_{N_4^*}\|_{X^{\alpha,\frac12+}},
\end{align*}
where
\[
L_3 = L_3(\alpha,d) : = f_0(\alpha,d) + f_1(\alpha,d) + f_2(\alpha,d) + f_3(\alpha,d) + \min(0,f_4(\alpha,d))  = (3\alpha - \frac{d}{2} -\frac{1}{2})+\min(0,   \frac{1}{2} + \alpha - \frac{d}{2}) -,
\]
and thus
\begin{equation*} 
    |I| \lesssim \frac{1}{N^{L_3 -}} E_2(u_0),
\end{equation*}
whenever $L_3->0$.

If now $N_1^* \sim N_2^* \gg N_3^*, N_4^*$, a similar analysis yields that  (\ref{I_last}) is
\begin{align*}
  & |I_{N_1,N_2,N_3,N_4}|  \lesssim \frac{1}{(N_1^*)^{L_2}} \|\D v_{N_1^*}\|_{X^{0,\frac12+}}  \|\D v_{N_2^*}\|_{X^{0,\frac12+}} \|v_{N_3^*}\|_{X^{\alpha,\frac12+}}
   \|v_{N_4^*}\|_{X^{\alpha,\frac12+}},
\end{align*}
where 
\begin{align*}
 L_2= L_2(\alpha,d)  := & f_0(\alpha,d) + f_1(\alpha,d) + f_2(\alpha,d) + \min(0,f_3(\alpha,d)) + \min(0,f_4(\alpha,d))   \\[4pt]
  = & (3\alpha - \frac{d}{2} -\frac{1}{2})+\min(0,   \frac{1}{2} + \alpha - \frac{d}{2}) -,   
\end{align*}
which further implies that
\begin{equation*} 
    |I| \lesssim \frac{1}{N^{ L_2 -}} E_2(u_0),
\end{equation*}
whenever $L_2->0$.

Assuming the same $\epsilon$-loss, we clearly have that
\[
L_4(\alpha,d)- \geq L_3(\alpha,d)- \geq L_2(\alpha,d)-,
\]
for all $\alpha \in (\frac{1}{2},1]$ and $d\geq 1$, and thus, if $L_2(\alpha,d)->0$ then $L_3(\alpha,d)->0$ and $L_4(\alpha,d)->0$.  This shows that we may only consider the frequency ordering case $N_1^* \sim N_2^* \gg N_3^*, N_4^*$ as any other case gives either the same or a better rate of decay for $I$, as well as a larger range for $\alpha$. Note that, in this particular example, $f_3(\alpha,d)\leq 0$ for $\alpha \in (\frac{1}{2},1]$ and $d\geq 1$, which is why $L_2 = L_3$, and hence the cases $N_1^* \sim N_2^* \sim N_3^* \gg N_4^*$ and $N_1^* \sim N_2^* \gg N_3^*, N_4^*$ give the same rate of decay for $I$ and the same range for $\alpha$.

The same argument holds for any number of frequencies $N_1, \ldots, N_k$, where $k \geq 2$ is even. To see this, first note that since $(\xi_1, \ldots, \xi_k) \in \Gamma_k$, we always have that $N_1^*\sim N_2^*$, thus there are $k-1$ cases to consider, specifically
\[
N_1^* \sim N_2^* \sim \ldots \sim N_m^* \gg N_{m+1}^*, \ldots, N_k^*, \qquad 2\leq m \leq k.
\]

Let us consider the general term
\begin{align*}
A := & \frac{1}{(N_1^*)^{g_0(\alpha,d)}} \int_{\tau_1+ \ldots + \tau_k=0} \int_{\Gamma_k}
\,|\reallywidetilde{\D v}(\tau_1,\xi_1)| \, |\reallywidetilde{ \chi \D \overline{v}}(\tau_2,\xi_2)|  \notag \\[6pt]
& \hspace{6cm} |\widetilde{v}(\tau_3,\xi_3)| 
\, \ldots 
\, |\widetilde{\overline{v}}(\tau_k,\xi_k)|\; d\Gamma_k(\xi)\, d\Gamma_k(\tau),    
\end{align*}
and define
\begin{align*}
A_{N_1,N_2,\ldots,N_k} := & \frac{1}{(N_1^*)^{g_0(\alpha,d)}} \int_{\tau_1+ \ldots + \tau_k=0} \int_{\Gamma_k}
\,|\reallywidetilde{\D v_{N_1^*}}(\tau_1,\xi_1)| \, |\reallywidetilde{ \chi \D \overline{v_{N_2^*}}}(\tau_2,\xi_2)|  \notag \\[6pt]
& \hspace{6cm} |\widetilde{v_{N_3^*}}(\tau_3,\xi_3)| 
\, \ldots 
\, |\widetilde{\overline{v_{N_k^*}}}(\tau_k,\xi_k)|\; d\Gamma_k(\xi)\, d\Gamma_k(\tau), \label{A_N1Nk}    
\end{align*}  
with $N_1^* \sim N_2^* \sim \ldots \sim N_m^* \gg N_{m+1}^*, \ldots, N_k^*$, for some $2\leq m \leq k$. Then, an analysis similar to the one above gives
\begin{align*}
  & |A_{N_1,N_2,\ldots,N_k}|  \\[6pt]
& \lesssim   \frac{1}{(N_1^*)^{g_0(\alpha,d)}} \;
\frac{ \|\D v_{N_1^*}\|_{X^{0,\frac12+}}  \|\D v_{N_2^*}\|_{X^{0,\frac12+}} \|v_{N_3^*}\|_{X^{\alpha,\frac12+}}
  \ldots   \|v_{N_k^*}\|_{X^{\alpha,\frac12+}}  }{(N_1^*)^{g_1(\alpha,d)} (N_2^*)^{g_2(\alpha,d)} \ldots   (N_k^*)^{g_k(\alpha,d)}}    \; 
 \\[6pt]
& \lesssim  \frac{1}{(N_1^*)^{g_0(\alpha,d)}} \;
\frac{\|\D v_{N_1^*}\|_{X^{0,\frac12+}}  \|\D v_{N_2^*}\|_{X^{0,\frac12+}} \|v_{N_3^*}\|_{X^{\alpha,\frac12+}}
  \ldots   \|v_{N_k^*}\|_{X^{\alpha,\frac12+}} }{(N_1^*)^{g_1(\alpha,d)} (N_1^*)^{g_2(\alpha,d)} \ldots (N_m^*)^{g_m(\alpha,d)}  (N_{m+1}^*)^{\min(0,g_{m+1}(\alpha,d))} \ldots (N_{k}^*)^{\min(0,g_{k}(\alpha,d))}  }     \\[6pt]
& \sim \frac{1}{(N_1^*)^{L_m}}  \|\D v_{N_1^*}\|_{X^{0,\frac12+}}  \|\D v_{N_2^*}\|_{X^{0,\frac12+}} \|v_{N_3^*}\|_{X^{\alpha,\frac12+}}
  \ldots   \|v_{N_k^*}\|_{X^{\alpha,\frac12+}} ,
\end{align*}
for some $g_0(\alpha,d), g_1(\alpha,d), \ldots, g_k(\alpha,d)  \in \R$,  where
\begin{align*} 
L_m = L_{m}(\alpha,d) & := g_0(\alpha,d) + g_1(\alpha,d) + g_2(\alpha,d)   + \ldots + g_{m}(\alpha,d)  \\[4pt]
& \hspace{5cm} + \min(0,g_{m+1}(\alpha,d))  + \ldots + \min(0,g_{k}(\alpha,d)) ,     
\end{align*}
with $2\leq m \leq k$, and thus
\begin{equation*} 
    |A| \lesssim \frac{1}{N^{L_m -}} E_2(u_0),
\end{equation*}
whenever $L_m->0$. In this case, assuming the same $\epsilon$-loss, we clearly have that
\[
 L_{ m}(\alpha,d)- \geq L_{2}(\alpha,d)-,
\]
for any $2 \leq m \leq k$, $\alpha \in (\frac{1}{2},1]$  and $d\geq 1$ . Thus, we may only consider the case $N_1^* \sim N_2^* \gg N_3^*, \ldots, N_k^*$ as any other case gives either the same or a better rate of decay, as well as a larger range for $\alpha$.  

The same argument holds if $A$ were instead given by
\begin{align*}
A := & \frac{1}{(N_1^*)^{g_0(\alpha,d)}} \int_{\tau_1+ \ldots + \tau_k=0} \int_{\Gamma_k}
\,|\reallywidetilde{ \chi \D v}(\tau_1,\xi_1)| \, |\reallywidetilde{  \D \overline{v}}(\tau_2,\xi_2)|  \notag \\[6pt]
& \hspace{6cm} |\widetilde{v}(\tau_3,\xi_3)| 
\, \ldots 
\, |\widetilde{\overline{v}}(\tau_k,\xi_k)|\; d\Gamma_k(\xi)\, d\Gamma_k(\tau).    
\end{align*}
In this case, the corresponding contribution $A_{N_1, N_2, \ldots, N_k}$ of $A$  is  equal to (\ref{II_example_term}) for $k=6$ and $g_0(\alpha,d) = 2\alpha$.

Taking this into account, by combining the estimates (\ref{I_2D}), (\ref{I_3D}), (\ref{II_2D}), and (\ref{II_3D}), the claim follows.
\end{proof}

\begin{proof}[Proof of Theorem \ref{Main_Theorem}]
  Theorem \ref{Main_Theorem} follows by an iteration argument. More specifically, by Lemma \ref{Lemma_3.5_fNLS}, we have that
\begin{equation} \label{iteration_bound}
E_2(u(t_0 + \delta)) \leq \left( 1 + \frac{C}{N^{\kappa-}} \right) E_2(u(t_0)), 
\end{equation}
where $\kappa = \frac{3}{2}(2\alpha - 1)$ for $d=2$, and  $\kappa = (4\alpha-3)$ for $d=3$. By iterating (\ref{iteration_bound}) $m  \sim N^{\kappa - }$ times and  using the elementary inequality $1+x \leq e^x$ for $x\geq 0$, we obtain 
\begin{equation*} 
E_2(u(t_0 + m \delta)) \leq C_\delta   E_2(u(t_0)),
\end{equation*}
where $C_\delta$ is independent of $N$. Hence, for a time interval of length $T \sim N^{\kappa -}$ we get
\[
\|\D u(T)\|_{L^2} \lesssim \|\D u_0\|_{L^2},
\]
which, by (\ref{D_scaling}), becomes
\[
\|u(T)\|_{H^s} \lesssim N^s \|u_0\|_{H^s}.
\]

Since $T \sim N^{\kappa -}$, we conclude that
\[
\|u(T)\|_{H^s} \lesssim T^{\frac{1}{\kappa} s+} \|u_0\|_{H^s} \lesssim (1+T)^{\frac{1}{\kappa} s+} \|u_0\|_{H^s}.
\]
This shows Theorem \ref{Main_Theorem} for large times. For small times, the result follows by local well-posedness, while the extension to negative times follows by time reversibility.
\end{proof}

\begin{remark}
    As in  \cite{Sohinger_Hartree_2D}, one may try to take advantage of the angle constraint that appears in our analysis by proving a 
    modified version of \cite[Lemma 8.2]{CKSTT_Resonant} in $X_{\alpha}^{s,b}$ spaces. However, such a result is only useful when the angle can be taken small, which is not the case for $\alpha \in (\frac{1}{2},1)$ and $d\geq 2$, as  (\ref{angle_constraint}) suggests.
\end{remark}

\section{Appendix}  \label{Appendix}  

\subsection{Proof of Lemma \ref{Sufficient_and_necessary_nonperiodic}} 

Before discussing the proof of Lemma \ref{Sufficient_and_necessary_nonperiodic}, let us first recall that (\ref{important_inequality}) is equivalent to
\begin{equation} \label{important_inequality_equivalent}
\left| |\xi_1|^{2\alpha} - |\xi_2|^{2\alpha} + |\xi_3|^{2\alpha} - |\xi_{123}|^{2\alpha} \right| \gtrsim \frac{|\xi_{12}|\; |\xi_{23}|}{(|\xi_1| + |\xi_2| + |\xi_3|)^{2-2\alpha}},  
\end{equation}
where $\alpha \in (\frac{1}{2},1]$ and $(\xi_1, \xi_2, \xi_3, \xi_{123}) \in \Omega_{nr}$, and $\xi_{123}$ is defined in (\ref{xi_ijpk}). Since the problem arises when $\xi_{12}$ and $\xi_{23}$ are orthogonal, it is natural to first seek such a counterexample. Indeed,  for
\begin{equation} \label{equation_counterexample}
\begin{cases}
    \xi_1  = (x,y), \\
    \xi_2  = (\zeta_1,\zeta_2), \\
    \xi_3  = (-\zeta_2,-\zeta_1), \\
\end{cases} 
\quad \quad \text{ such that }  \quad \quad  
\begin{cases}
    x-y = \zeta_2-\zeta_1, \\
    y \neq -\zeta_2,  \\
    \zeta_1\neq \zeta_2,
\end{cases}     
\end{equation}
we have $(\xi_1+\xi_2) \cdot (\xi_2 +\xi_3)  = 0$, and $|\xi_1| = |\xi_{123}|, |\xi_2| = |\xi_3|$, and therefore the LHS of  (\ref{important_inequality_equivalent}) is zero. On the other hand
\[
|\xi_1+\xi_2| \, |\xi_2+\xi_3| = (\sqrt{2} |\zeta_2+y|) (\sqrt{2} |\zeta_1-\zeta_2|) =  2 |\zeta_2+y|\, |\zeta_1-\zeta_2|,
\]
which is non-zero by assumption, and hence the RHS of (\ref{important_inequality_equivalent}) is non-zero. 

Note that the example (\ref{equation_counterexample}) (as well as the example (\ref{equation_counterexample_2}) below) applies also to higher dimensions if one extends the vectors $\xi_1, \xi_2$ and $\xi_3$ with zero entries.  Moreover,  observe that, in example (\ref{equation_counterexample}), $(\xi_1+\xi_2)$ and $(\xi_2 +\xi_3)$ are orthogonal, but (\ref{important_inequality_equivalent}) can fail even when they are not orthogonal. Consider for instance
\begin{equation} \label{equation_counterexample_2}
    \xi_1  = (M,0), \qquad 
    \xi_2  = (M,2M), 
    \qquad \xi_3  = (-2M+\epsilon,-M),    
\end{equation}
where $0<\epsilon \ll 1 \ll M$ and $\epsilon$ is fixed. Then, it is not hard to see that for any $\alpha>0$,
\[ 
\frac{\textup{LHS of }(\ref{important_inequality_equivalent})}{\textup{RHS of }(\ref{important_inequality_equivalent})}
\ \longrightarrow\ 0
\qquad \text{as }  \qquad M\to\infty.
\]
Indeed, the LHS and RHS of (\ref{important_inequality_equivalent}) scale like $M^{2\alpha -1}$ and $M^{2\alpha}$, respectively, thus their quotient has asymptotic order $M^{-1} \ra 0$, as $M \ra + \infty$.

Therefore, it is natural to investigate whether a positive lower bound on $|\cos \angle(\xi_1+\xi_2, \xi_2 +\xi_3)|$ is required for (\ref{important_inequality_equivalent}) to hold. In the case $\alpha =1$, it is easy to see that (\ref{important_inequality_equivalent}) holds with $|\cos \angle(\xi_1+\xi_2, \xi_2 +\xi_3) | \geq \gamma$ if  $\gamma>0$. This follows since
\[
 |\xi_1|^{2} - |\xi_2|^{2} + |\xi_3|^{2} - |\xi_1 + \xi_2 + \xi_3|^{2}  = - 2 (\xi_1+\xi_2) \cdot (\xi_2 + \xi_3),
\]
for any $\xi_1, \xi_2, \xi_3 \in \R^d$. It is not clear, however, if this is also the case when $\alpha <1$.

We continue with the proof of Lemma \ref{Sufficient_and_necessary_nonperiodic}. The main ingredient is the double mean value theorem, which says that, for any $f \in C^2(\R^d)$, we have
\begin{equation} \label{double_MVT}
f(\xi+ \eta + \lambda) - f(\xi + \eta) - f(\xi +\lambda) + f(\xi) =  \int_0^1 \int_0^1  \eta^{\mathsf T} \nabla^2f(\xi + s\lambda + t \eta) \lambda \; ds \, dt.    
\end{equation}
The proof of (\ref{double_MVT}) follows by applying the one-dimensional mean value theorem twice. In our application, however, the function we are interested in, namely $x \ra |x|^{2\alpha}$, where $\alpha \in (\frac{1}{2},1]$, is not twice differentiable at $0$. Thus, a weaker version of the double mean value theorem is needed that treats such singularities. First, we need the following two facts.

\begin{lemma} \label{i_and_ii}
\begin{enumerate}[
    label=\textbf{ (\roman*)}, ref=(\roman*), leftmargin=0.7cm]
    \item \label{Fact_1} Let $d\geq 1$, let $\beta\in[0,1)$, and let $\eta\in\R^d\setminus\{0\}$. Then, for every $y\in\R^d$,
\[
\int_0^1 |y+s\eta|^{-\beta}\,ds
\lesssim  |\eta|^{-\beta}.
\]
\item \label{Fact_2} Let $d\geq 1$, let $u:[0,1]\to\R^d$ be continuous, and assume that there exists $c\in[0,1]$ such that $u\in C^1([0,1]\setminus\{c\})$ and $
u'\in L^1(0,1)$, where $u'$ is the weak derivative of $u$. Then
\[
u(1)-u(0)=\int_0^1 u'(s)\,ds.
\]
\end{enumerate}
\end{lemma}

 \begin{proof}
We first show (i). We write 
\[
y=y_\perp + r \frac{\eta}{|\eta|},
\qquad y_\perp\cdot \frac{\eta}{|\eta|}=0,
\qquad r\in\R.
\]
Then
\[
|y+s\eta|^2 = |y_\perp|^2 + |r+s|\eta||^2 \geq |r+s|\eta||^2.
\]
Therefore,
\[
|y+s\eta|^{-\beta}
\leq |\eta|^{-\beta}\,\left| s+\frac{r}{|\eta|}\right|^{-\beta},
\]
and thus
\begin{equation} \label{i_inequality}
\int_0^1 |y+s\eta|^{-\beta}\,ds
\leq |\eta|^{-\beta}\int_0^1 |s+a|^{-\beta}\,ds,    
\end{equation}
where $a:=r/|\eta|$. Using that $\beta \in [0,1)$, it is then easy to show that the integral in (\ref{i_inequality}) is bounded by $\frac{2}{1-\beta}$. This  proves (i).

For (ii), we may assume without loss of generality $d=1$. 
We consider the case $c\in(0,1)$, with $c \in \{0,1\}$ following similarly. Now, for small enough $\epsilon>0$, we have
\[
u(c-\epsilon)-u(0)=\int_0^{c-\epsilon} u'(s)\,ds, \qquad u(1)-u(c+\epsilon)=\int_{c+\epsilon}^1 u'(s)\,ds.
\]
Hence
\begin{equation} \label{ii_first}
u(1)-u(0)
=
\int_0^{c-\epsilon} u'(s)\,ds
+
\int_{c+\epsilon}^1 u'(s)\,ds
+
[u(c+\epsilon)-u(c-\epsilon)].    
\end{equation}

Since $u$ is continuous at $c$, we have
\begin{equation} \label{ii_1}
u(c+\epsilon)-u(c-\epsilon)\to 0
\qquad\text{as }\epsilon\to 0.    
\end{equation}

Also, since $u'\in L^1(0,1)$,
\begin{equation} \label{ii_2}
\int_0^{c-\epsilon} u'(s)\,ds + \int_{c+\epsilon}^1 u'(s)\,ds
\to \int_0^1 u'(s)\,ds
\qquad\text{as }\epsilon\to 0.    
\end{equation}
The claim then follows by letting $\epsilon \ra 0$ in (\ref{ii_first}) and using (\ref{ii_1}) and (\ref{ii_2}).
\end{proof}

 Using  Lemma \ref{i_and_ii}, we can show the following version of the double mean value theorem.

\begin{proposition}[Double MVT with singularity]
\label{double_MVT_with_singularity}
Let $d\geq 1$ and $f:\R^d \ra \R$. Assume that
\begin{itemize}
\item[a)] $f\in C^1(\R^d)\cap C^2(\R^d\setminus\{0\})$,
\item[b)] there exist $C>0$, $r_0>0$, and $\beta\in[0,1)$ such that
\begin{equation}
\|\nabla^2f(x)\|\leq C|x|^{-\beta},
\qquad\text{whenever } \enspace 0<|x|<r_0. \label{double_MVT_Corollary_D2_condition}    
\end{equation}
\end{itemize}
Then, for every $\xi,\eta,\lambda\in\R^d$,
\[
f(\xi+\eta+\lambda)-f(\xi+\eta)-f(\xi+\lambda)+f(\xi)  =
\int_0^1\int_0^1  
\eta^{\mathsf T}  \nabla^2f(\xi+t\lambda+s\eta) \lambda\; ds\, dt,
\]
with the convention that the integrand is set equal to $0$ whenever
$\xi+t\lambda+s\eta=0$.
\end{proposition}

\begin{proof}

First, if $\eta=0$ or $\lambda=0$, the claim follows trivially, thus we may assume that $\eta,
\lambda\neq 0$. Now, set
\[
\Gamma:=\{\xi+t\lambda+s\eta:\ (s,t)\in[0,1]^2\}.
\]

Since $\Gamma$ is compact and $f \in C^2(\R^d \setminus \{ 0\})$, 
\[
M:=\sup\bigl\{\|\nabla^2f(x)\|:\ x\in\Gamma,\ |x|\geq r_0\bigr\} < \infty.
\]
Therefore, by (\ref{double_MVT_Corollary_D2_condition}), for every $x\in\Gamma\setminus\{0\}$ we have
\begin{equation} \label{D2_bound}
\|\nabla^2f(x)\|
\leq
M + C|x|^{-\beta}.    
\end{equation}

Now, fix $t\in[0,1]$, and define
\[
u_t(s):=\nabla f(\xi+t\lambda+s\eta),
\qquad s\in[0,1].
\]
Since $f\in C^1(\R^d)$, the map $u_t$ is continuous on $[0,1]$. Because $\eta\neq 0$, for each fixed $t \in [0,1]$ the equation $\xi+t\lambda+s\eta=0$ has at most one solution in $[0,1]$. If no such solution exists, we can deduce immediately that
\begin{equation} \label{equation_u_t}
\nabla f(\xi+t\lambda+\eta)-\nabla f(\xi+t\lambda)
=
\int_0^1 \nabla^2f(\xi+t\lambda+s\eta)\, \eta \; ds  
\end{equation}
by the fundamental theorem of calculus.  If there exists such $c_t \in [0,1]$, then
$u_t\in C^1([0,1]\setminus\{c_t\};\R^d)$, and for every $s\in[0,1]\setminus\{c_t\}$,
\[
u_t'(s)=\nabla^2f(\xi+t\lambda+s\eta)\,\eta.
\]

Therefore, by (\ref{D2_bound}) and Lemma \ref{i_and_ii} (i), we deduce
\begin{align*}
\int_0^1 |u_t'(s)|\,ds
&\leq M|\eta| + C|\eta|\int_0^1 |\xi+t\lambda+s\eta|^{-\beta}\,ds \lesssim  |\eta| + |\eta|^{1-\beta} < \infty,
\end{align*}
hence $u_t'\in L^1(0,1)$. Thus, by Lemma \ref{i_and_ii} (ii), we conclude again (\ref{equation_u_t}). 

Now, set
\[
v(t):=f(\xi+\eta+t\lambda)-f(\xi+t\lambda),
\qquad t\in[0,1].
\]
Since $f\in C^1(\R^d)$, we have $v\in C^1([0,1])$, and 
\[
v'(t)=\int_0^1  
\eta^{\mathsf T} \nabla^2f(\xi+t\lambda+s\eta) \lambda\; ds
\]
for every $t \in [0,1]$. From the one-dimensional fundamental theorem of calculus, it follows that
\[
v(1)-v(0) = f(\xi+\eta+\lambda)-f(\xi+\eta)-f(\xi+\lambda)+f(\xi)  =
\int_0^1\int_0^1  
\eta^{\mathsf T} \nabla^2f(\xi+t\lambda+s\eta) \lambda \; ds \, dt.
\]
This completes the proof.
\end{proof}

We can now use Proposition \ref{double_MVT_with_singularity} to prove Lemma \ref{Sufficient_and_necessary_nonperiodic}.

\begin{proof}[Proof of Lemma \ref{Sufficient_and_necessary_nonperiodic}]

 We may always assume that $\eta,\lambda \neq 0$ and $\alpha \in (\frac{1}{2},1)$, otherwise the inequality is trivial. For simplicity, we first assume $d=2$. Let $p := 2\alpha$ and  $f(x)=|x|^p$, $x \in \R^d$. By Proposition  \ref{double_MVT_with_singularity}, we have
    \[
   \Phi = \Phi(\xi,\eta,\lambda) := f(\xi + \eta + \lambda) - f(\xi + \eta) - f(\xi + \lambda) + f(\xi) = \int_0^1  \int_0^1 \eta^{\mathsf T}   \nabla^2 f (\zeta(t,s))   \lambda \; ds d t,
    \]
where $\zeta := \zeta(t,s) := \xi+ t \eta + s \lambda$. Since
\[
\nabla^2 f(\zeta)=p\,|\zeta|^{p-2}\pt{I+(p-2) \frac{\zeta \zeta^{\mathsf T}}{|\zeta|^2}}, \qquad \zeta \neq 0, 
\]
 we may write $\Phi  = p c(\eta \cdot \lambda - (2-p) \eta^{\mathsf T} B \lambda)$, where
\[
c :=  \int_0^1  \int_0^1 |\zeta(t,s)|^{p-2} \; ds \, dt \; \in \R \quad  \text{ and } \quad   B :=  \frac{1}{c} \int_0^1 \int_0^1 \frac{\zeta(t,s) \zeta^{\mathsf T}(t,s) |\zeta(t,s)|^{p-2}}{|\zeta(t,s)|^2} \; ds \, dt \;  \in M_{2\times 2}(\R).
\]

We now analyse the term $\eta^{\mathsf T} B \lambda$. Let $\widehat{\zeta} := \frac{\zeta}{|\zeta|}$. For each $(t,s) \in [0,1]^2$ we have
\[
|(\eta\cdot\widehat\zeta(t,s))(\lambda\cdot\widehat\zeta(t,s))|
\leq \sup_{|u|=1}|(\eta\cdot u)(\lambda\cdot u)|.
\]
Therefore,
\begin{equation} \label{eta_B_lambda_sup}
 |\eta^{\mathsf T}B\lambda|
\leq \frac{1}{c} \int_0^1 \int_0^1|(\eta\cdot\widehat\zeta)(\lambda\cdot\widehat\zeta)| |\zeta|^{p-2} \; ds \, dt
\leq \sup_{|u|=1}|(\eta\cdot u)(\lambda\cdot u)|.    
\end{equation}

To estimate the RHS of (\ref{eta_B_lambda_sup}), we may assume without loss of generality that 
\[
\eta=|\eta|(1,0)\qquad \text{ and } \qquad \lambda=|\lambda|(\cos\theta,\sin\theta),
\]
where $\theta := \cos^{-1}(\gamma_0) \in [-\frac{\pi}{2}, \frac{\pi}{2}]$ and  $\gamma_0 = \gamma_0(\eta, \lambda) := \frac{|\eta \cdot \lambda|}{|\eta| \, |\lambda|}$. Let $u \in \R^2$ with $|u|=1$, and write 
\[
u = (\cos\nu_u, \sin\nu_u).
\]
Then
\[
\eta \cdot u = |\eta| \cos\nu_u, \qquad   \lambda\cdot u
= |\lambda|\big(\cos\theta\cos\nu_u + \sin\theta\sin\nu_u\big)
= |\lambda|\cos(\nu_u-\theta),
\]
hence
\begin{align*}
|(\eta\cdot u)(\lambda\cdot u)| =  |\eta| \, |\lambda| \, |\cos(\nu_u) \cos(\nu_u - \theta)|
 = \frac{|\eta| \, |\lambda|}{2} \,| \cos(\theta) + \cos(2\nu_u - \theta) |.
\end{align*}
This quantity attains its maximum when $\nu_u$ is chosen so that $\cos(2\nu_u - \theta) = 1$, therefore
\[
\sup_{|u|=1}|(\eta\cdot u)(\lambda\cdot u)| = \frac{1+\gamma_0}{2} |\eta| \, |\lambda|,
\]
which then implies that 
\[ 
(2-p) |\eta^{\mathsf T} B \lambda| \leq (2-p) \frac{1+\gamma_0}{2} |\eta| \, |\lambda|.
\]

Therefore, for $S= S(\xi,\eta,\lambda):= |\xi|+|\eta|+|\lambda|$ we obtain
\begin{align}
|\Phi| & \geq p\,c\Big(|\eta\cdot\lambda|-(2-p)\,|\eta^{\mathsf T}B\lambda|\Big) \notag \\[6pt]
& \geq p\,S^{\,p-2}\Big(|\eta\cdot\lambda|-(2-p)\frac{1+\gamma_0(\eta,\lambda)}{2}\,|\eta|\,|\lambda|\Big) \notag \\[6pt]
& \geq p\,S^{\,p-2}\,|\eta|\,|\lambda|\;\frac{p\gamma-(2-p)}{2}, \label{Omega_last}
\end{align}
where the last line follows from $\gamma_0(\eta,\lambda)\geq \gamma$. For (\ref{Omega_last}) to be positive, we need
\[
\displaystyle \frac{p\gamma-(2-p)}{2} > 0 \quad\Longleftrightarrow\quad \gamma> \frac{1-\alpha}{\alpha},
\]
 as desired. This shows the sufficiency of (\ref{angle_constraint}). 

 Conversely, for the sake of contradiction, suppose that there exists $\gamma_0 \in (0,1]$ with $\gamma_0 \leq \frac{1-\alpha}{\alpha}$ such that (\ref{important_inequality}) holds under (\ref{vector_angle_constraint}) with $\gamma = \gamma_0$. For $\nu \in [-\frac{\pi}{2},\frac{\pi}{2}]$ to be determined later, define
\[
A(\nu):=\begin{pmatrix}\cos\nu & -\sin\nu\\[4pt]\sin\nu & \cos\nu\end{pmatrix},\quad
\eta :=\begin{pmatrix}1\\[2pt]0\end{pmatrix},\quad
\lambda=\begin{pmatrix}\gamma_0\\[2pt]\sigma\end{pmatrix},\quad
\sigma := \sqrt{1-\gamma_0^2}, \quad
\xi_n := A(\nu)\begin{pmatrix}n\\[2pt]0\end{pmatrix}
=\begin{pmatrix}n\cos\nu\\[2pt]n\sin\nu\end{pmatrix}.
\]

 For $(t,s)\in[0,1]^2$ define
\[
\zeta_n(t,s) := \xi_n+t\eta+s\lambda
= \begin{pmatrix} n\cos\nu + t + s\gamma_0\\[4pt] n\sin\nu + s\sigma\end{pmatrix},
\]
and additionally set
\[
c_n := \int_0^1 \int_0^1 |\zeta_n(t,s)|^{\,p-2}\,ds\,dt,
\qquad
B_n := \frac{1}{c_n} \int_0^1 \int_0^1 |\zeta_n(t,s)|^{\,p-2}\,\widehat{\zeta_n}(t,s)\widehat{\zeta_n}(t,s)^{\mathsf T}\,ds\,dt,
\]
where $\widehat{\zeta_n}:=\frac{\zeta_n}{|\zeta_n|}$. Finally, let
\[
D_n := \gamma_0 - (2-p)\eta^{\mathsf T}B_n\lambda  , \qquad \Phi_n \;:=\; p\,c_n D_n, \qquad  S_n := |\xi_n|+|\eta|+|\lambda|.
\]

To obtain a contradiction, we need to show that  $|\Phi_n| S_n^{2-p} \ra 0$. First, we have
\[
\zeta_n(t,s) = n v(\nu) + O(1),\quad \text{ where } \quad v(\nu):=(\cos\nu,\sin\nu),
\]
uniformly in $(t,s) \in [0,1]^2$, hence
\[
|\zeta_n(t,s)|^{p-2}=n^{p-2}(1+o(1)), 
\qquad \text{ and } \qquad
\widehat{\zeta_n}(t,s)\widehat{\zeta_n}(t,s)^{\mathsf T}
=
v(\nu)v(\nu)^{\mathsf T}+o(1),
\]
uniformly in $(t,s) \in [0,1]^2$. Moreover,
\begin{equation} \label{cn_convergence}
 c_n = \iint_{[0,1]^2} |\zeta_n(t,s)|^{p-2}\,ds\,dt = n^{p-2}\big(1+o(1)\big),   
\end{equation}
thus  $B_n = v(\nu)v(\nu)^{\mathsf T}+o(1)$ in operator norm, and therefore
\begin{equation}  \label{lambdaBneta_convergence}
\eta^{\mathsf T}B_n\lambda\ra (\eta\cdot v(\nu))(\lambda\cdot v(\nu)) = \cos\nu\big(\gamma_0\cos\nu+\sigma\sin\nu\big).    
\end{equation}

Let $G(\nu) := \cos\nu\big(\gamma_0\cos\nu+\sigma\sin\nu\big)$. Setting $(\gamma_0,\sigma)=(\cos\phi,\sin\phi)$ with $\phi=\arctan(\sigma/\gamma_0)$ and using elementary trigonometric identities, we can write
\begin{equation} \label{formula_G}
G(\nu)=\frac{\gamma_0}{2}+\frac12\cos(2\nu-\phi).    
\end{equation}
Pick $\nu \in [- \frac{\pi}{2}, \frac{\pi}{2}]$ such that
\begin{equation} \label{G_equation}
 \gamma_0 - (2-p)G(\nu)=0 \quad \iff \quad G(\nu)=\frac{\gamma_0}{2-p}.   
\end{equation}
Note that, by (\ref{formula_G}), a solution of (\ref{G_equation}) exists if and only if
\[
\left|\gamma_0\cdot\frac{p}{2-p}\right|\leq 1 \quad\Longleftrightarrow\quad \gamma_0 \leq \frac{1-\alpha}{\alpha},
\]
as assumed. By (\ref{lambdaBneta_convergence}) and the choice of $\nu$, we have $D_n = o(1)$. Thus, by (\ref{cn_convergence}) and since in addition $S_n^{2-p} = O(n^{2-p})$, we deduce that
\[
|\Phi_n| S_n^{2-p} = p \, c_n \, S_n^{2-p}\, |D_n| = p \, n^{p-2} \,(1+o(1)) \, O(n^{2-p}) \, o(1) =  O(1) \, o(1) = o(1),
\]
which contradicts  (\ref{important_inequality}), as desired. Thus, we have shown the necessity of (\ref{angle_constraint}).  

For higher dimensions, the proof of sufficiency is identical, while for necessity, we define
\[
A(\nu)=
\begin{pmatrix}
\cos\nu & -\sin\nu & 0\\
\sin\nu & \cos\nu & 0\\
0& 0&I_{d-2}
\end{pmatrix},
\quad \eta=e_1,\quad \lambda=\gamma_0 e_1+\sigma e_2, \qquad
\xi_n=A(\nu)(n e_1),
\]
where
\[
e_1 := (1,0, \ldots, 0), \qquad e_2 := (0,1,0,\ldots, 0),
\]
and obtain a contradiction by the same argument. This completes the proof.
\end{proof}

\subsection{The Kato-Ponce inequality} In this section, we prove Proposition \ref{KP_extension_Hs_Linfty}.  The proof follows from the classical Kato-Ponce inequality for Schwartz functions \cite{GrafakosOh, DongLi} and a density argument. A similar result was shown in \cite[Theorem 1.3]{BadrBernicotRuss} in a more general setting.

First, by \cite[Theorem 1.9, (1.15)]{DongLi} with $p_1=p_2=2$ and $q_1=q_2=+\infty$ we have
\begin{equation}\label{KP_Schwartz}
\||\nabla|^\sigma(fg)\|_{L^2}
\lesssim 
\|f\|_{L^\infty}\||\nabla|^\sigma g\|_{L^2}
+
\||\nabla|^\sigma f\|_{L^2}\|g\|_{L^\infty},
\end{equation}
for $f,g \in \mathcal{S}(\R^d)$ and $\sigma\geq 0$.  To extend (\ref{KP_Schwartz}) to functions in $H^\sigma(\R^d) \cap L^\infty(\R^d)$ we need the following approximation lemma.

\begin{lemma}\label{Hs_Linfinity_approximation}
Let $\sigma\geq 0$ and  $h\in H^\sigma(\R^d)\cap L^\infty(\R^d)$. Then there exists a sequence $h_n\in C_c^\infty(\R^d)$ such that
\begin{equation}\label{Hs_limit}
h_n\to h
\qquad\text{in }H^\sigma(\R^d),
\end{equation}
and
\begin{equation}\label{Hs_Linfinity_bounded}
\|h_n\|_{L^\infty}\leq \|h\|_{L^\infty}
\qquad\text{for all } \quad n\geq 1.
\end{equation}
\end{lemma}

\begin{proof}
Choose $\chi,\rho \in C_c^\infty(\R^d)$ such that
\[
0\leq \chi\leq 1,
\qquad
\chi(x)=1 \quad\text{for } |x|\leq 1, \qquad
\chi(x)=0 \quad\text{for } |x|\geq 2, \qquad  \rho\geq 0,
\qquad
\int_{\R^d}\rho(x)\,dx=1,
\]
and let us define 
\[
\chi_n(x):=\chi(x/n), \qquad \rho_n(x):=n^d\rho(nx), \qquad \text{and} \qquad h_n:=\rho_n*(\chi_n h).
\]

By definition, $h_n\in C_c^\infty(\R^d)$. Moreover, for every $x \in \R^d$, we have that
\[
\begin{aligned}
|h_n(x)|
&=
\left|
\int_{\R^d}\rho_n(y)\chi_n(x-y)h(x-y)\,dy
\right|  \leq
\|h\|_{L^\infty}
\int_{\R^d}\rho_n(y)\,dy =
\|h\|_{L^\infty}.
\end{aligned}
\]
 This proves (\ref{Hs_Linfinity_bounded}).

We now prove (\ref{Hs_limit}). First, we have 
\begin{equation}\label{hn_h_decomposition}
\|h_n-h\|_{H^\sigma}
\leq
\|\rho_n*(\chi_n h)-\chi_n h\|_{H^\sigma}
+
\|\chi_n h-h\|_{H^\sigma}.
\end{equation}
We show that both terms on the RHS of (\ref{hn_h_decomposition}) tend to zero. We begin with the cutoff term, namely, we prove that
\begin{equation}\label{cutoff_convergence}
\chi_n h\to h
\qquad\text{in } \enspace H^\sigma(\R^d).
\end{equation}

First, we claim that
\begin{equation}\label{uniform_cutoff_bound}
\|\chi_n u\|_{H^\sigma}
\leq
C_\sigma \|u\|_{H^\sigma}
\qquad
\text{for all }u\in H^\sigma(\R^d).
\end{equation}
Indeed, using the identity $\widehat{\chi_n u}
= c_d \,\widehat{\chi_n}*\widehat u$, the fact that $\langle \xi\rangle^\sigma
\lesssim_\sigma \langle \eta\rangle^\sigma\langle \xi-\eta\rangle^\sigma$ and Minkowski's inequality, we obtain
\begin{align*}
  \|\chi_n u\|_{H^\sigma}  & \lesssim_\sigma \left\|
\int_{\R^d} \langle \eta\rangle^\sigma|\widehat{\chi_n}(\eta)| \langle \xi -\eta\rangle^\sigma|\widehat u(\xi-\eta)|\,d\eta \right\|_{L^2_\xi} \lesssim_\sigma
\int_{\R^d}
\langle \eta\rangle^\sigma|\widehat{\chi_n}(\eta)|
\left\|
\langle \xi-\eta\rangle^\sigma \widehat u(\xi-\eta)
\right\|_{L^2_\xi}
\,d\eta  \\[6pt]
& \sim_\sigma
\left(
\int_{\R^d}
\langle \eta\rangle^\sigma|\widehat{\chi_n}(\eta)|\,d\eta
\right)
\|u\|_{H^\sigma} \lesssim_\sigma  \|u\|_{H^\sigma},  
\end{align*}
 for any $u \in H^\sigma(\R^d)$, where $C_\sigma>0$ is independent of $n$. This proves \eqref{uniform_cutoff_bound}. 

Now, fix $\epsilon>0$ and let $\varphi\in C_c^\infty(\R^d)$
such that
\[
\|h-\varphi\|_{H^\sigma}<\epsilon.
\]
Since $\varphi$ has compact support, $\chi_n\varphi=\varphi$ for large $n$. Therefore,
\[
\limsup_{n\to\infty}\|\chi_n h-h\|_{H^\sigma}
\leq \limsup_{n\to\infty} ( \|\chi_n(h-\varphi)\|_{H^\sigma}
+
\|\chi_n\varphi-\varphi\|_{H^\sigma}
+
\|\varphi-h\|_{H^\sigma}) \leq
(C_\sigma+1)\epsilon.
\]
Since $\epsilon>0$ was arbitrary, we conclude (\ref{cutoff_convergence}). 

Next, we prove that
\begin{equation}\label{mollification_term}
\|\rho_n*(\chi_n h)-\chi_n h\|_{H^\sigma}\to 0.
\end{equation}
First note that, for any $u\in H^\sigma(\R^d)$, 
\begin{equation}\label{fixed_mollification}
\qquad 
\|\rho_n*u\|_{H^\sigma}
\leq
\|u\|_{H^\sigma}, \qquad\text{ and }  \qquad \rho_n*u\to u
\quad\text{in }H^\sigma(\R^d).
\end{equation}

By (\ref{fixed_mollification}) and (\ref{cutoff_convergence}), we deduce that
\begin{align}
\|\rho_n*(\chi_n h)-\chi_n h\|_{H^\sigma}
& \leq
\|\rho_n*((\chi_n-1)h)\|_{H^\sigma}
+
 \|\rho_n*h-h\|_{H^\sigma}
+
\|(\chi_n-1)h\|_{H^\sigma} \\[6pt]
& \leq \|\rho_n*h-h\|_{H^\sigma}
+
2\|(\chi_n-1)h\|_{H^\sigma} \ra 0.
\end{align}
 This proves (\ref{mollification_term}). By (\ref{cutoff_convergence}) and (\ref{mollification_term}), we conclude that $h_n \ra h$ in $H^\sigma(\R^d)$. This completes the proof.
\end{proof}

\begin{proof}[Proof of Proposition \ref{KP_extension_Hs_Linfty}]

For $\sigma=0$, the claim follows trivially by H\"{o}lder's inequality, thus we may assume that $\sigma>0$. By Lemma~\ref{Hs_Linfinity_approximation}, there exist $f_n,g_n\in C_c^\infty(\R^d)$ such that
\begin{equation}\label{fn_gn_Hs_convergence}
f_n\to f,
\qquad
g_n\to g
\qquad\text{in }H^\sigma(\R^d),
\end{equation}
and
\begin{equation}\label{fn_gn_Linfty_bound}
\|f_n\|_{L^\infty}\leq \|f\|_{L^\infty},
\qquad
\|g_n\|_{L^\infty}\leq \|g\|_{L^\infty}
\qquad\text{for every } n \geq 1.
\end{equation}

By (\ref{KP_Schwartz}) and (\ref{fn_gn_Linfty_bound}), we obtain 
\begin{equation*}
\||\nabla|^\sigma(f_n g_n)\|_{L^2} \lesssim \|f\|_{L^\infty}\||\nabla|^\sigma g_n\|_{L^2} + \||\nabla|^\sigma f_n\|_{L^2}\|g\|_{L^\infty}.
\end{equation*}

Let $u_n:=f_n g_n$ and $ u:=fg$. By Plancherel's theorem, $\widehat{u_n}\to \widehat u$  in $L^2(\R^d)$, thus we may pass to a further subsequence such that $\widehat{u_{n_k}}(\xi)\to \widehat u(\xi)$ for a.e. $\xi\in\R^d$. By Fatou's lemma,

\begin{align}
\||\nabla|^\sigma u\|_{L^2}^2 = \int_{\R^d}
|\xi|^{2\sigma}|\widehat u(\xi)|^2\,d\xi \leq \liminf_{k\to\infty} \int_{\R^d} |\xi|^{2\sigma}|\widehat{u_{n_k}}(\xi)|^2\,d\xi
= \liminf_{k\to\infty} \||\nabla|^\sigma u_{n_k}\|_{L^2}^2. \label{Fatou_inequality}
\end{align}
Therefore, by (\ref{Fatou_inequality}) and (\ref{KP_Schwartz}) we have
\[
\begin{aligned}
\||\nabla|^\sigma(fg)\|_{L^2} &\leq \liminf_{k\to\infty}\||\nabla|^\sigma(f_{n_k} g_{n_k})\|_{L^2} \lesssim \liminf_{k\to\infty} (
\|f\|_{L^\infty}\||\nabla|^\sigma g_{n_k}\|_{L^2} + \||\nabla|^\sigma f_{n_k}\|_{L^2}\|g\|_{L^\infty}).
\end{aligned}
\]
Since $n_k\ra \infty$ and using (\ref{fn_gn_Hs_convergence}), we deduce (\ref{KP_Hs_Linfty}). Note that this also shows that $fg \in H^\sigma(\R^d)$. This completes the proof.
\end{proof}

\subsection{Proof of Proposition \ref{Proposition_3.1_d.alpha}} 

In this section, we prove Proposition \ref{Proposition_3.1_d.alpha}. Unlike the standard linear Schr\"{o}dinger equation with $\alpha=1$, there is no exact Galilean invariance in the fractional case $\alpha \in (\frac{1}{2},1)$. This was also observed in \cite[page 4]{HongSire}. Therefore, unlike the proof of  \cite[Proposition 3.1]{Sohinger_Hartree_2D},  which considers (\ref{equation_fractional_hartree}) for $\alpha=1$ and relies on  Galilean invariance, we bound the Duhamel operator directly, using the Kato-Ponce inequality  (\ref{KP_Hs_Linfty}).

As in \cite{Sohinger_Hartree_2D}, we use the following localised estimates in $X^{s,b}$.  For $f\in C_0^\infty(\R)$ and   $\delta,b,b' \in \R$ such that
\[
0<\delta\leq1, \quad 0< b'<\frac{1}{2}<b<1, \quad b+b' <1,
\]
it holds that
\begin{align}
    \| f(t/\delta)\,S_\alpha(t)u_0\|_{X^{s,b}} & \lesssim \,\delta^{\frac{1}{2}-b} \|u_0\|_{H^s}, \label{equation_141}  \\[6pt]
      \Big\| f(t/\delta)\int_{t_0}^t S_\alpha(t-t')\,w(t')\,dt' \Big\|_{X^{s,b}}
& \lesssim \,\delta^{1-b-b'} \|w\|_{X^{s,-b'}},  \label{Lemma_A.2.11_Dinh_PhD} \\[6pt]
    \| f(t/\delta) w \|_{X^{s,-b'}} &\lesssim \|w\|_{X^{s,-b'}}, \label{equation_144}
\end{align}
where $S_\alpha$ is defined in (\ref{Salpha_definition}) and the implicit constants depend only on $f,s,b, b'$, but not on $t_0,\delta$ or $\alpha$.
The proof of (\ref{equation_141}) is carried out exactly as \cite[Lemma 3.1]{KenigPonceVega_1993}. For (\ref{Lemma_A.2.11_Dinh_PhD}), see  \cite[Lemma A.2.11]{DinhPhD}, and for  (\ref{equation_144}), see  \cite[Lemma 2.11]{TaoDispersive}. 

We also fix $\chi,\phi\in C_0^\infty(\R)$ with $\chi,\phi \geq 0 $ such that
\begin{align*}
    \chi = 1 \enspace  \text{ on } \enspace  [-1,1], \quad \chi = 0 \enspace  \text{ outside } \enspace  [-2,2], \\[4pt]
    \phi = 1 \enspace  \text{ on } \enspace  [-2,2], \quad \phi = 0 \enspace  \text{ outside } \enspace  [-4,4],
\end{align*}
and denote $\chi_\delta(t) := \chi(t/\delta)$ and $\phi_\delta(t) := \phi(t/\delta)$. We adopt this notation in what follows.

\begin{proof}[Proof of Proposition \ref{Proposition_3.1_d.alpha}]
 We argue by a fixed-point argument. We may assume without loss of generality that 
$t_0 = 0$. We consider
\begin{equation*} \label{equation_125}
Lw= L_\delta w := \chi_\delta(t)S_\alpha(t)u_0 - i\chi_\delta(t)\int_0^t S_\alpha(t-t')(V*|w_\delta|^2)w_\delta(t')\,dt',
\end{equation*}
where $w_\delta := \phi_\delta w$. For $s>\frac{d}{2}$, let
\[
B := \big\{ w : \ \|w\|_{X^{\alpha,b}} \leq 2c\,\delta^{\frac{1}{2}-b} \|u_0\|_{H^{\alpha}}, 
\ \|w\|_{X^{s,b}} \leq 2c\,\delta^{\frac{1}{2}-b} \|u_0\|_{H^{s}} \big\},
\]
where $c>0$ is the implicit constant in (\ref{equation_141}).
Arguing as in  \cite[Proposition 3.2]{Sohinger_S1}, we deduce that $B$ is complete with respect to the $\|\cdot\|_{X^{\alpha,b}}$-norm. For $w\in B$ we have 
\begin{equation*} \label{equation_126}
\|Lw\|_{X^{\alpha,b}} \leq c \, \delta^{\frac{1}{2}-b} \, \|u_0\|_{H^\alpha} + c_1 \, \delta^{1-b-b'}  \,  \|(V*|w_\delta|^2)w_\delta\|_{X^{\alpha,-b'}},
\end{equation*}
and analogously
\begin{equation*} \label{equation_127}
\|\mathcal DLw\|_{X^{0,b}} \leq c \, \delta^{\frac{1}{2}-b}  \, \| \mathcal Du_0\|_{L^2} + c_1  \,\delta^{1-b-b'}\,\|\mathcal D((V*|w_\delta|^2)w_\delta)\|_{X^{0,-b'}}.
\end{equation*}

We now estimate  $\|(V*|w_\delta|^2)w_\delta\|_{X^{s,-b'}}$. Let us define 
\[
\mathcal{N}(f):=(V*|f|^2)f.
\]
Choose $0<b'<\frac{1}{2} < b$ such that $b+b' <1$ and
\begin{equation} \label{b_condition}
 2-3b-b'>0.
\end{equation}
Note that such $b,b'$ can be chosen arbitrarily close to $\frac{1}{2}$. Let us also define $\rho'>1$  by $\frac{1}{\rho'} = \frac{1}{2} + b'$, let $\rho>1$ with $\frac{1}{\rho}+\frac{1}{\rho'} = 1$, and set $p=2\rho'$. By Sobolev embedding in time, we have
\begin{equation} \label{Sobolev_embeddings}
L_t^{\rho'}\hookrightarrow H_t^{-b'}.    
\end{equation}

By Corollary \ref{Strichartz_Xsb_embedding}, we also have
\begin{equation} \label{Corollary_for_Y}
\|f\|_{L_t^pL_x^\infty}\lesssim \|f\|_{X^{\alpha,b}},
\end{equation}
for $\alpha>\frac d2-\frac{2\alpha}{p}$. Since $p$ can be chosen arbitrarily close to $2$ from above,  $\alpha$ can be taken arbitrarily close to $\frac{d}{4}$ in (\ref{Corollary_for_Y}).

We claim that 
\begin{equation} \label{Nonlinearity_bound}
    \|\mathcal N(w_\delta)\|_{X^{\alpha,-b'}}
\lesssim
\|V\|_{L^1}\,
\|w\|_{X^{\alpha,b}}^3.
\end{equation}
First, note that  $\mathcal N(w_\delta) = \phi_\delta^3 \mathcal N(w)$, thus (\ref{equation_144}) implies that 
\begin{equation} \label{Ndelta_to_N}
\|\mathcal N(w_\delta)\|_{X^{\alpha,-b'}} \lesssim \|\mathcal N(w)\|_{X^{\alpha,-b'}}.  
\end{equation}
Since $w \in B$ and $\alpha < s$, we have that $w \in X^{s,b} \subset L_t^\infty H_x^s \subset L_t^\infty H_x^\alpha$, thus $w(t) \in H_x^s \subset H_x^\alpha$ for any $t$. By (\ref{Corollary_for_Y}), we also have that $w \in L_t^p L_x^\infty$, hence $w(t) \in L_x^\infty$ for a.e. $t$. Therefore, it follows by the Kato-Ponce inequality that for a.e. $t$,
\begin{equation}  \label{Nonlinearity_bound_product_Hs}
    \|\mathcal N(w)\|_{H_x^\alpha}
=
\|(V*|w|^2)w\|_{H_x^\alpha}
\lesssim
\|V*|w|^2\|_{L_x^\infty}\|w\|_{H_x^\alpha}
+
\|V*|w|^2\|_{H_x^\alpha}\|w\|_{L_x^\infty}.
\end{equation}
By Young's inequality,
\[ 
\|V*|w|^2\|_{L_x^\infty}
\leq
\|V\|_{L^1}\|w\|_{L_x^\infty}^2.
\]

Moreover, since $\| \widehat{V}\|_{L^\infty} \leq \|V\|_{L^1}$, using again the Kato-Ponce inequality  (\ref{KP_Hs_Linfty}) we get that, for a.e. $t$,
\begin{equation} \label{V_convolution_w2}
\|V*|w|^2\|_{H_x^\alpha}
\leq
\|V\|_{L^1}\||w|^2\|_{H_x^\alpha}
\lesssim
\|V\|_{L^1}\|w\|_{L_x^\infty}\|w\|_{H_x^\alpha}.    
\end{equation}
Substituting (\ref{V_convolution_w2}) into (\ref{Nonlinearity_bound_product_Hs}), we obtain that, for a.e. $t$,
\begin{equation}\label{Nonlinearity_bound_2}
\| \mathcal N(w)(t)\|_{H_x^\alpha}
\lesssim
\|V\|_{L^1}\,
\|w(t)\|_{L_x^\infty}^2
\|w(t)\|_{H_x^\alpha}.
\end{equation}
Taking the $L_t^{\rho'}$ norm in (\ref{Nonlinearity_bound_2}), and using the embeddings (\ref{Sobolev_embeddings}) and (\ref{Corollary_for_Y}), we get
\begin{equation} \label{N_embedding}
\|\mathcal N(w)\|_{X^{\alpha,-b'}} \leq \|\mathcal N(w)\|_{L_t^{\rho'}H_x^\alpha}
\lesssim
\|V\|_{L^1}\,
\|w\|_{L_t^pL_x^\infty}^2
\|w\|_{L_t^\infty H_x^\alpha} \leq \|V\|_{L^1}\,
\|w\|_{X^{\alpha,b}}^3,    
\end{equation}
which, together with (\ref{Ndelta_to_N}), implies (\ref{Nonlinearity_bound}). 

By the same analysis, we can show
\begin{equation}\label{Nonlinearity_bound_difference}
\|\mathcal N(w_\delta)-\mathcal N(z_\delta)\|_{X^{\alpha,-b'}}
\lesssim \|V\|_{L^1}\,
( \|w\|_{X^{\alpha,b}}^2+\|z\|_{X^{\alpha,b}}^2 ) \,
\|w-z\|_{X^{\alpha,b}}
\end{equation}
by using (\ref{equation_144}), the Kato-Ponce inequality  (\ref{KP_Hs_Linfty}), Young's inequality,  and the identities
\[
\mathcal N(w)- \mathcal N(z)
=
(V*|w|^2)(w-z)
+
\bigl(V*(|w|^2-|z|^2)\bigr)z
\]
and 
\[
|w|^2-|z|^2
=
(w-z)\overline{w}
+
z\,\overline{(w-z)}.
\]
Therefore, we deduce that
\begin{align*}
\|Lw\|_{X^{\alpha,b}}
&\leq
c\, \delta^{\frac12-b}\, \| u_0\|_{H^\alpha}
+
\, c_1 \, \delta^{1-b-b'} \, \|\mathcal N(w_\delta)\|_{X^{\alpha,-b'}} \\[4pt]
& \leq 
c \, \delta^{\frac12-b}\, \| u_0\|_{H^\alpha}
+
c_2 \, \delta^{\frac12-b} \, \delta^{\frac{1}{2}-b'} \, \|w\|_{X^{\alpha,b}}^3  \\[4pt]
& \leq  c\delta^{\frac12-b}\| u_0\|_{H^\alpha} + c_3\,  \delta^{\frac12-b}\, \delta^{2-3b-b'}\, \| u_0\|_{H^\alpha}^3.
\end{align*}

Arguing as in (\ref{N_embedding}), we also have that
\[
\|\mathcal N(w)\|_{X^{s,-b'}} \leq \|\mathcal N(w)\|_{L_t^{\rho'}H_x^s}
\lesssim
\|V\|_{L^1}\,
\|w\|_{L_t^pL_x^\infty}^2
\|w\|_{L_t^\infty H_x^s} \leq \|V\|_{L^1}\,
\|w\|_{X^{\alpha,b}}^2 \, \|w\|_{X^{s,b}},
\]
which gives us, in addition, the estimate
\begin{align*}
\|Lw\|_{X^{s,b}}  \leq  c \,\delta^{\frac12-b} \,\| u_0\|_{H^s} + c_4 \, \delta^{\frac12-b}\, \delta^{2-3b-b'} \, \| u_0\|_{H^\alpha}^2 \,\| u_0\|_{H^s}.
\end{align*}
Note that $2-3b-b'>0$ by (\ref{b_condition}). Thus, we can choose $\delta>0$ small enough such that
\[
\delta^{\,2-3b-b'} \, \| u_0\|_{H^\alpha}^2\leq c,
\]
 and so that $\delta = \delta(s,E(u_0), M(u_0))$ depends continuously on $E(u_0)$ and $M(u_0)$. With such a choice of $\delta$,  $L$ maps $B$ into itself.

By (\ref{Nonlinearity_bound_difference}) we also have that
\[
\|Lw - Lz\|_{X^{\alpha,b}} \leq c_5 \, \delta^{1-b-b'} \, \delta^{1-2b}\,  \|u_0\|_{H^\alpha}^2 \|w - z\|_{X^{\alpha,b}},
\]
thus, by taking $\delta>0$ possibly even smaller so that 
\begin{equation*} \label{delta_second_choice}
c_5 \, \delta^{1-b-b'} \, \delta^{1-2b} \,\|u_0\|_{H^\alpha}^2 \leq \frac{1}{2},    
\end{equation*}
the map $L$ becomes a contraction on $B$. By Banach's fixed-point theorem, there exists a unique
$v\in B$ such that $Lv=v$. In particular,
\begin{equation} \label{v_bound}
\|v\|_{X^{\alpha,b}}
\lesssim \delta^{\frac12-b} \, \|u_0\|_{H^\alpha},
\end{equation}
which proves (\ref{47.d.alpha}). Since $\chi_\delta = 1$ and $\phi_\delta =  1$ on $[0,\delta]$, $v$ satisfies (\ref{equation_fractional_hartree}) on $[0,\delta]$. Thus, arguing as in the proof of  \cite[Proposition 3.1]{Sohinger_Hartree_2D} by invoking the Sobolev embedding $H^s \hookrightarrow L^\infty$ with $s> \frac{d}{2}$ and Gr\"{o}nwall's inequality,\footnote{The use of Gr\"{o}nwall's inequality is justified since, in fact, $v \in C_t H^s_x$ by the algebra property of $H^s$, with $s>\frac{d}{2}$.} we deduce that
\[
v|_{[0,\delta]} = u|_{[0,\delta]}.
\]

We have shown (\ref{46.d.alpha}) and (\ref{47.d.alpha}).  For (\ref{48.d.alpha}), first note that $\D(\mathcal N(v_\delta)) = \phi_\delta^3 \D(\mathcal N(v))$, thus (\ref{equation_144}) implies
\[
\|\D(\mathcal N(v_\delta))\|_{X^{0,-b'}}
\lesssim
\|\D(\mathcal N(v))\|_{X^{0,-b'}}.
\]
By Lemma \ref{Kato_Ponce_for_D}, Young's inequality and the identity $\D(V*|v|^2)=V* \D(|v|^2)$, we deduce that, for a.e. $t$,
\begin{align*}
\| \D(\mathcal N(v))\|_{L_x^2} &\lesssim \|V*|v|^2\|_{L_x^\infty}\|\D v\|_{L_x^2} + \|\D(V*|v|^2)\|_{L_x^2}\|v\|_{L_x^\infty} \\[4pt]
&\lesssim \|V\|_{L^1}\|v\|_{L_x^\infty}^2\|\D v\|_{L_x^2}
+ \|V\|_{L^1}\|\D(|v|^2)\|_{L_x^2}\|v\|_{L_x^\infty} \\[4pt]
&\lesssim \|V\|_{L^1}\|v\|_{L_x^\infty}^2\|\D v\|_{L_x^2}.
\end{align*}

Consequently, arguing as in (\ref{N_embedding}) by taking the $L_t^{\rho'}$ norm and using the embeddings (\ref{Sobolev_embeddings}) and (\ref{Corollary_for_Y}), we obtain
\begin{equation*}
\|\mathcal D(\mathcal N(v_\delta))\|_{X^{0,-b'}}
\lesssim
\|V\|_{L^1}\,
\|v\|_{X^{\alpha,b}}^2
\|\mathcal Dv\|_{X^{0,b}}.
\end{equation*}
Therefore, we deduce
\begin{align*}
\|\mathcal Dv\|_{X^{0,b}}
&\leq
c \, \delta^{\frac12-b} \,\|\mathcal Du_0\|_{L^2}
+
c_1 \, \delta^{1-b-b'} \, \|\mathcal D(\mathcal N(v_\delta))\|_{X^{0,-b'}} \\[4pt]
&\leq
c \, \delta^{\frac12-b} \, \|\mathcal Du_0\|_{L^2}
+
c_6 \, \delta^{1-b-b'}\, \|v\|_{X^{\alpha,b}}^2 \, \|\mathcal Dv\|_{X^{0,b}}.
\end{align*}

Using (\ref{v_bound}) and taking $\delta>0$ possibly even smaller, we get 
\begin{equation} \label{Dv_delta}
c_6 \delta^{1-b-b'}\|v\|_{X^{\alpha,b}}^2
\leq
c_7 \,\delta^{2-3b-b'} \, \|u_0\|_{H^\alpha}^2 < 1.    
\end{equation}
Since $v \in X^{s,b}$, it follows that $\D v \in X^{0,b}$. Thus, by (\ref{Dv_delta}) we deduce that 
\[
\|\mathcal Dv\|_{X^{0,b}}
\lesssim
\delta^{\frac12-b}\|\mathcal Du_0\|_{L^2},
\]
which proves (\ref{48.d.alpha}). This completes the proof.
\end{proof}

\begin{remark}
 The assumption $s> \frac{d}{2}$ in (\ref{equation_fractional_hartree}) is only used to prove uniqueness in Proposition \ref{Proposition_3.1_d.alpha}. Although this condition may be relaxed, we do not address this question here.
\end{remark}